\documentclass[reqno]{amsart}
\usepackage{amssymb}
\usepackage[usenames, dvipsnames]{color}
\usepackage{pxfonts}
\usepackage{enumerate}
\usepackage{amsmath}
\usepackage[driverfallback=dvipdfm]{hyperref}{\tiny}
\usepackage{verbatim}

\allowdisplaybreaks[4]

\numberwithin{equation}{section}

\newtheorem{theorem}{Theorem}[section]
\newtheorem{corollary}[theorem]{Corollary}
\newtheorem{lemma}[theorem]{Lemma}
\newtheorem{proposition}[theorem]{Proposition}

\theoremstyle{definition}

\theoremstyle{definition}

\theoremstyle{definition}

\theoremstyle{definition}

\def\C {{\mathbb C}}
\usepackage{todonotes}

\def\R {\mathbb{R}}
\def\N {\mathbb{N}_0}
\def\Z {\mathbb{Z}}
\def\d{{\,\rm d}}

\begin{document}

 \title{Log-type ultra-analyticity of elliptic equations with gradient terms}

\author[H. Dong]{Hongjie Dong}
\address[H. Dong]{Division of Applied Mathematics, Brown University, 182 George Street, Providence, RI 02912, USA}
\email{Hongjie\_Dong@brown.edu}
\thanks{H. Dong was partially supported by the NSF under agreement DMS-2350129.}

\author[M. Wang]{Ming Wang}
\address{School of Mathematics and Statistics, HNP-LAMA, Central South University, Hunan, Changsha, 410083, PR China} 	
\email{m.wang@csu.edu.cn}
\thanks{M. Wang is partially supported by the National Natural Science Foundation of China under grants 12171442 and 12171178.}

\date{\today}

\begin{abstract}
It is well known that every solution of an elliptic equation is analytic if its coefficients are analytic. However, less is known about the ultra-analyticity of such solutions. This work addresses the problem of elliptic equations with lower-order terms, where the coefficients are entire functions of exponential type. We prove that every solution satisfies a quantitative logarithmic ultra-analytic bound and demonstrate that this bound is sharp. The results suggest that the ultra-analyticity of solutions to elliptic equations cannot be expected to achieve the same level of ultra-analyticity as the coefficients.
\end{abstract}
\maketitle




\section{Introduction and the main results}

Recall that a function \( u: \Omega \subset \mathbb{R}^d \to \mathbb{C} \) is called analytic at a point \( x \in \Omega \) if there exist constants \( C, R > 0 \) such that \( |\partial_x^\alpha u(x)| \leq \frac{C \alpha!}{R^{|\alpha|}} \) for every multi-index \( \alpha \in \mathbb{N}_0^d \), where \(\mathbb{N}_0 = \{0, 1, 2, \ldots\}\) denotes the set of nonnegative integers. 
Here, and throughout, we use the convention \( \alpha! = \alpha_1! \alpha_2! \cdots \alpha_d! \) and \( |\alpha| = \alpha_1 + \alpha_2 + \cdots + \alpha_d \) if \( \alpha = (\alpha_1, \alpha_2, \ldots, \alpha_d) \in \mathbb{N}_0^d \). A function is analytic in a domain \( \Omega \) if it is analytic at every point in \( \Omega \).

The analyticity of solutions to elliptic equations is a classical topic in the literature. For example, every harmonic function in \( \Omega \) (satisfying \( \Delta u = 0 \)) is analytic in \( \Omega \). More generally, if \( L \) is an elliptic differential operator with analytic coefficients, then solutions to \( Lu = 0 \) are also analytic; see, for instance, \cite{John1955, Morrey-CPAM1957}. In certain contexts, one may require quantitative analytic bounds (with explicit constants \( C \) and \( R \)) to address various problems of interest. For example, such bounds on eigenfunctions of the Laplacian have been used in \cite{DonnellyNodalsets1988} to tackle Yau's conjecture on analytic manifolds, which concerns the Hausdorff measure of the nodal sets of eigenfunctions. Moreover, these analytic bounds are also applied to establish observability inequalities for heat equations with measurable observation sets; see \cite{ZhangCJEMS2014, WangAnalyticityCOCV2023}.

The quantity \(R\) is called the analytic radius of a function \(u\) at the point \(x\). This means that \(u\) can be extended to an analytic function \(u(z)\) in the ball \(\{z \in \mathbb{C}^d : |z - x| < R\}\). There has been extensive work devoted to finding quantitative bounds for the analytic radius of solutions to various PDEs; for example, see \cite{CheminOntheradius2020,DongOptimallocal2009,Herbst-Advmath2011,Titi-JFA2000} for results on the Navier-Stokes equations. If \(R = \infty\), then the function \(u\) is analytic in the entire space \(\mathbb{C}^d\), and in this case we call the function ultra-analytic. Typical examples of such functions are those whose Fourier transforms exhibit exponential decay at infinity:
\begin{align}\label{equ-four-decay}
\widehat{u}(\xi) = O(e^{-|\xi|^s}),
\end{align}
where \(s > 1\). In fact, one can show that such a function satisfies higher-order derivative bounds of the form:
\begin{align}\label{equ-intro-ana-1}
\|\partial_x^\alpha u\|_{L^\infty(\mathbb{R}^d)} \leq C M^{|\alpha|} (\alpha!)^{\delta}, \quad \forall \alpha \in \mathbb{N}_0^d,
\end{align}
for some constants \(C, M > 0\), where \(\delta = 1/s \in (0,1)\). 

Functions satisfying \eqref{equ-four-decay} play a significant role in the uncertainty principle in harmonic analysis; see \cite{EKPV-Bull.Amer.Math2012}. For instance, the classical Hardy uncertainty principle \cite{Hardy1933} states that any non-trivial function satisfying \eqref{equ-four-decay} with \(s = 2\) cannot decay faster than \(e^{-A|x|^2}\) as \(x \to \infty\) when \(A > 0\) is sufficiently large. Based on the uncertainty principle of ultra-analytic functions, several unique continuation inequalities have been established for Schrödinger equations in \cite{WWZ-JEMS} and for KdV equations in \cite{LiWangsiam2021}. 

Motivated by these results, we are interested in establishing ultra-analytic bounds for solutions to elliptic equations, such as \(\Delta u = Vu\) where \(V\) is ultra-analytic. Surprisingly, to the best of our knowledge, this problem has not been previously addressed. A closely related result is that the eigenfunctions of the harmonic oscillator \(H = -\Delta + |x|^2\) satisfy the ultra-analytic bound \eqref{equ-intro-ana-1} with \(\delta = 1/2\); see \cite{JamingSpectralestimates2021}. In other words, the solution to \(\Delta u = Vu\), where \(V = |x|^2 - \lambda\) and \(\lambda\) is an eigenvalue, satisfies the ultra-analytic bound \eqref{equ-intro-ana-1}. 

Initially, we hypothesized that \eqref{equ-intro-ana-1} would hold if \(V\) satisfies strong growth rates on its derivatives, for example if \(V\) is an entire function of exponential type. However, it turns out that this is not the case. In this work, we show that the solution only satisfies a log-type ultra-analytic bound, and furthermore, we demonstrate that this bound is sharp.

To state our result, we consider a more general elliptic equation that includes a gradient term:
\begin{align}\label{equ-elip}
\Delta u = W(x) \nabla u + V(x) u, \quad x \in \mathbb{R}^d,
\end{align}
where \(W: \mathbb{R}^d \to \mathbb{C}^d\) and \(V: \mathbb{R}^d \to \mathbb{C}\) are entire functions of exponential type \(C_0\) for some constant \(C_0 > 0\), namely,

\begin{align}\label{equ-V}
\sup_{x \in \mathbb{R}^d} |\partial_x^\alpha W(x)|,\, \sup_{x \in \mathbb{R}^d} |\partial_x^\alpha V(x)| \leq C_0^{|\alpha|}, \quad \forall \alpha \in \mathbb{N}_0^d.
\end{align}

We first present our result for the one-dimensional case:
\begin{theorem}[\(d = 1\)]\label{thm-1}
Let \(p \in [1, \infty]\) and \(u \in L^p(\mathbb{R})\) be a solution to \eqref{equ-elip} with \(\|u\|_{L^p(\mathbb{R})} = 1\). Then, there exists a constant \(K = K(p) > 0\), depending only on \(p\), such that the solution satisfies the following log-type ultra-analytic upper bound: for any \(\kappa > 1\),
\begin{align}\label{equ-ultra-ana}
\|\partial_x^n u\|_{L^p(\mathbb{R})} \leq \left( \kappa C_0 + K\left(\frac{1}{\log \kappa} + 1\right) \right)^n \frac{n!}{\log^n(n+e)}, \quad \forall n \in \mathbb{N}_0.
\end{align}
\end{theorem}

The estimate \eqref{equ-ultra-ana} is stronger than typical analytic estimates, as the fixed analytic radius \( R \) is effectively enhanced by \( C \log n \). This is why we refer to it as a log-type ultra-analytic bound. The bound \eqref{equ-ultra-ana} is sharp in two key aspects. 
First, if \(C_0 > 0\) is fixed, then for any \(\lambda > 1\) and \(C > 0\), the estimate \eqref{equ-ultra-ana} cannot be improved to:
\begin{align}\label{equ-intro-ana-2}
\|\partial_x^n u\|_{L^p(\mathbb{R})} \leq C^n \frac{n!}{\log^{\lambda n}(n+e)}, \quad \forall n \in \mathbb{N}_0.
\end{align}
Notice that \eqref{equ-intro-ana-1} implies \eqref{equ-intro-ana-2} for any \(\lambda > 0\). Therefore, \eqref{equ-intro-ana-1} is not expected to hold for generic solutions to \eqref{equ-elip}.
Second, when \(C_0 \to \infty\), the condition \(\kappa > 1\) is essentially sharp. More precisely, the bound \eqref{equ-ultra-ana} cannot be strengthened to:
\begin{align}\label{equ-intro-ana-3}
\|\partial_x^n u\|_{L^p(\mathbb{R})} \leq \left( \kappa C_0 + C \right)^n \frac{n!}{\log^n(n+e)}, \quad \forall n \in \mathbb{N}_0
\end{align}
for any \(\kappa < 1\). In fact, we construct an explicit solution to \eqref{equ-elip} that demonstrates the failure of estimates \eqref{equ-intro-ana-2}-\eqref{equ-intro-ana-3}; see Section 2.2 for details.

The sharpness of Theorem \ref{thm-1} shows that the ultra-analyticity of solutions to elliptic equations is inherently weaker than the ultra-analyticity of the coefficients in the lower-order terms. This observation appears to be novel. In fact, every solution of \eqref{equ-elip} has at least the same regularity as that of coefficients if we are working in usual analytic spaces or Gevrey classes.

It is a natural extension to consider whether Theorem \ref{thm-1} can be applied to second-order elliptic equations with variable leading coefficients. Specifically, we consider the equation:
\begin{align}\label{equ-91-1}
    \partial_x(A(x) \partial_x u) = W(x) \partial_x u + V(x) u, \quad x \in \mathbb{R},
\end{align}
where \( A, W, V \) are entire functions of exponential type, satisfying similar bounds to those in \eqref{equ-V}. We will show that the solutions to \eqref{equ-91-1} satisfy the log-type ultra-analytic bounds like those in \eqref{equ-ultra-ana} if and only if \( A \) is a constant multiple of the exponential function \( e^{iCx} \) for some \( C \in \mathbb{R} \); see Proposition \ref{prop-div}. However, in this scenario, equation \eqref{equ-91-1} can be reformulated into the form of \eqref{equ-elip} with new choices for \( W \) and \( V \). In other words, the ultra-analyticity of solutions to \eqref{equ-91-1} holds only in this special case, which reduces to an elliptic equation with constant leading coefficients.

Theorem \ref{thm-1} is established using a delicate induction argument, which also provides a similar result for normalized \(L^2(\mathbb{R})\) solutions to the equation:
\begin{align}\label{equ-frac-ellip-D}
|D| u = W(x) u,
\end{align}
where \( W \) is the same as above, and \( |D| \) is the fractional differential operator with the Fourier symbol \( |\xi| \). While the solution to the elliptic equation \( |D| u = W(x) u \) can be viewed as a stationary solution of the parabolic equation:
\begin{align}\label{equ-frac-heat-D}
\partial_t u + |D| u = W(x) u,
\end{align}
there is a fundamental difference between them. For the fractional heat equation \eqref{equ-frac-heat-D}, the solution exhibits Fourier decay at most like \( e^{-t|\xi|} \) (considering the free case where \( W = 0 \)) for generic initial data. Consequently, by the Paley-Wiener theorem, the solution has an analytic radius of \( t \) and cannot achieve the ultra-analyticity present in the elliptic case.

Now we turn to the ultra-analyticity in higher dimensions. 
\begin{theorem}[Ultra-Analyticity in Higher Dimensions, \(d \geq 2\)]\label{thm-2}  
Let \( p \in [1, \infty] \) and \( u \in L^p(\mathbb{R}^d) \) be a solution to \eqref{equ-elip} with \( \|u\|_{L^p(\mathbb{R}^d)} = 1 \).

\begin{itemize}  
  \item[(i)] If \( p = 1 \), then there exists a constant \( K = K(d) > 0 \), depending only on \( d \), such that  
  \begin{align}\label{equ-ultra-ana-L1}  
  \|\partial_x^\alpha u\|_{L^1(\mathbb{R}^d)} \leq K \left( \kappa C_0 + K \left( \frac{1}{\log \kappa} + 1 \right) \right)^{|\alpha|} \frac{|\alpha|!}{\log^{|\alpha|}(|\alpha| + e)}, \quad \forall \alpha \in \mathbb{N}_0^d.  
  \end{align}
  
  \item[(ii)] If $1<p\leq \infty$, then there exists a constant \( K = K(d, p) > 0 \), depending only on \( d \) and \( p \), such that for any \( \kappa > 1 \), the solution satisfies
  \begin{align}\label{equ-ultra-ana-Lp}  
  \|\partial_x^\alpha u\|_{L^p(\mathbb{R}^d)} \leq \left( \kappa C_0 + K \left( \frac{1}{\log \kappa} + 1 \right) \right)^{|\alpha|} \frac{|\alpha|!}{\log^{|\alpha|}(|\alpha| + e)}, \quad \forall \alpha \in \mathbb{N}_0^d.  
  \end{align}
  
\end{itemize}  
\end{theorem}  

Theorem \ref{thm-2} gives similar ultra-analytic bounds in higher dimensions as in the one-dimensional case. However, for some technical reasons,  the proof will be split into three cases, \(1<p<\infty\), \(p=1\), and \(p=\infty\). 
In the case where \(1 < p < \infty\), the proof relies on maximal regularity for elliptic equations in \(L^p(\mathbb{R}^d)\) and involves an induction argument. For the case \(p = 1\), since maximal regularity fails, a different approach is required. Using local maximal regularity in \(L^{p_0}\) for \(1 < p_0 < \infty\), we first establish a similar ultra-analyticity result as in \eqref{equ-ultra-ana-Lp} in $l^1L^{p_0}(\R^d)$, the amalgams of $l^1$ and $L^{p_0}(\R^d)$.
Furthermore, we observe that every normalized \(L^1(\mathbb{R}^d)\) solution to \eqref{equ-elip} also belongs to \(l^1 L^{p_0}(\mathbb{R}^d)\) for some \(p_0 \in (1, \infty)\), with its norm being bounded by a constant depending only on \(d\). By combining these two facts, we deduce \eqref{equ-ultra-ana-L1}. This strategy, however, does not apply to the case \(p = \infty\). Fortunately, every \(L^\infty(\mathbb{R}^d)\) solution to \eqref{equ-elip} must be H\"{o}lder continuous in \(\mathbb{R}^d\) due to the elliptic regularity theory. Thus, we employ Schauder estimates and iterate in H\"{o}lder spaces, and then we obtain \eqref{equ-ultra-ana-Lp}  with \(p=\infty\).

Throughout the paper, we use \(A \lesssim B\) to denote that \(A \leq CB\) for some numerical constant \(C > 0\). If the constant \(C\) depends on variables \(x, y, \ldots\), we write \(A \lesssim_{x,y,\ldots} B\). If both \(A \lesssim B\) and \(B \lesssim A\) hold, we write \(A \sim B\).

\section{one-dimensional case}

In this section, we first give the proof of Theorem \ref{thm-1} and then show the sharpness of it.
\subsection{Proof of Theorem \ref{thm-1}}

Let
\begin{align}\label{equ-s3-1}
a_j=\frac{1}{l!\log^j (j+e)}, \quad 0\leq j\leq n, j+l=n.
\end{align}
We first study the monotonicity of $a_j$.

\begin{lemma}\label{lem-mono}
Let $a_j$ be given by \eqref{equ-s3-1}. Then when $n$ is sufficiently large we have
\begin{align*}
&a_j\leq a_{j+1},  \quad \mbox{ if }n/2< j\leq n-\log n-2, \\
&a_j\geq a_{j+1},  \quad \mbox{ if }j\geq n-\log n +1.
\end{align*}
\end{lemma}
\begin{proof}
By \eqref{equ-s3-1}, $a_j\leq a_{j+1}$ is equivalent to
\begin{align}\label{equ-s3-2}
j+\frac{\log^{j+1}(j+1+e)}{\log^j(j+e)}\leq n.
\end{align}
We need to compute the asymptotic order of the term in \eqref{equ-s3-2}.
Since
$$
\log \left(1+ \frac{1}{j+e}\right)=\frac{1}{j}+O(\frac{1}{j^2}), \quad
\frac{1}{\log(j+e)}
=\frac{1}{\log j}+O(\frac{1}{j\log^2 j}),
$$
we deduce that
$$
\frac{\log(j+1+e)}{\log(j+e)} = 1+\frac{1}{\log(j+e)}  \log \left(1+ \frac{1}{j+e}\right) = 1+\frac{1}{j\log j}+O(\frac{1}{j^2\log j}).
$$
Thus, we have
\begin{align*}
\left( \frac{\log(j+1+e)}{\log(j+e)} \right)^j&=\exp\left\{j\log \frac{\log(j+1+e)}{\log(j+e)}\right\} =\exp\left\{j\log (1+\frac{1}{j\log j}+O(\frac{1}{j^2\log j}))\right\}\\
&=\exp \left\{j (\frac{1}{j\log j}+ O(\frac{1}{j^2\log j})   \right\}=1+\frac{1}{\log j}+O(\frac{1}{\log^2 j}).
\end{align*}
This, together with $\log(j+1+e)=\log j+O(\frac{1}{j})$, gives that
\begin{align}\label{equ-s3-7}
\log(j+1+e)\left( \frac{\log(j+1+e)}{\log(j+e)} \right)^j=\log j+1+O(\frac{1}{\log j}).
\end{align}
It follows from \eqref{equ-s3-2} and \eqref{equ-s3-7} that $a_j\leq a_{j+1}$
is equivalent to
\begin{align}\label{equ-s3-8}
j+\log j+1+O(\frac{1}{\log j})\leq n.
\end{align}
Now we show that if $n/2<j\leq n-\log n-2$, then \eqref{equ-s3-8} holds.
Noting that $O(\frac{1}{\log j})\to 0$ as $n\to \infty$, \eqref{equ-s3-8} 
follows from  $j\leq n-\log n-2$.

Similarly, we show that if  $j\geq n-\log n+1$, then
\begin{align}\label{equ-s3-10}
j+\log j + 1 +O(\frac{1}{\log j})\geq n,
\end{align}
which is equivalent to $a_j\geq a_{j+1}$.
In fact, if $j\geq n-\log n+1$, then
$$
j+\log j + 1 +O(\frac{1}{\log j})\geq j+\log j\geq   n-\log n+1+\log(n-\log n+1)\geq n
$$
if $n$ is sufficiently large. Thus, the proof is complete.
\end{proof}

We also need the following upper bounds of $a_j$ when $j$ is close to $n-\log n$.

\begin{lemma}\label{lem-aj-bound-1}
If $j=n-\log n+s$ with $\frac{|s|}{(\log n)^{\frac{2}{3}}}=O(1)$ as $n\to \infty$, then
$$
a_j\lesssim \frac{n}{ (\log n)^{n+\frac{1}{2}} e^{\frac{3}{2}s^2/\log n}}.
$$
\end{lemma}
\begin{proof}
If $j=n-\log n+s$, then $a_j$ can be rewritten as
\begin{align}\label{equ-aj-1}
a_j\leq \frac{1}{(\log n-s)!\log^j (j+e)}<\frac{1}{(\log n-s)!\log^j j}.
\end{align}
We need to estimate the two factors on the right-hand side.

First, by the Stirling formula $n!\sim \sqrt{n}(\frac{n}{e})^n$, we find
$$
(\log n-s)!\sim \sqrt{\log n-s}(\frac{\log n-s}{e})^{\log n-s}\sim \sqrt{\log n}(\frac{\log n}{e})^{\log n-s}( {1-\frac{s}{\log n}} )^{\log n-s}.
$$
Note that when $\frac{|s|}{(\log n)^{\frac{2}{3}}}=O(1)$,
\begin{align*}
( {1-\frac{s}{\log n}} )^{\log n-s}&=e^{(\log n-s)\log(1-\frac{s}{\log n})}=e^{(\log n-s)\Big(-\frac{s}{\log n}+\frac{1}{2}(\frac{s}{\log n})^2-\frac{1}{3}(\frac{s}{\log n})^3+o((\frac{s}{\log n})^3)\Big)}\\
&=e^{-s+\frac{3}{2}\frac{s^2}{\log n}+O(\frac{s^3}{(\log n)^2})}\sim  e^{-s+\frac{3}{2}\frac{s^2}{\log n}},
\end{align*}
so we have
\begin{align}\label{equ-aj-2}
(\log n-s)!\sim \sqrt{\log n}(\frac{\log n}{e})^{\log n-s}e^{-s+\frac{3}{2}\frac{s^2}{\log n}}=
 ( \log n )^{\log n-s+\frac{1}{2}}e^{-\log n+\frac{3}{2}\frac{s^2}{\log n}}.
\end{align}

Similarly, we have
\begin{align*}
\log^j j&=(\log (n-\log n+s))^{n-\log n+s}=\Big(\log n+\log(1-\frac{\log n-s}{n})\Big)^{n-\log n+s}\\
&=\Big(\log n- \frac{\log n-s}{n}+o(\frac{1}{n})\Big)^{n-\log n+s}\\
&=(\log n)^{n-\log n+s}\Big(1- \frac{\log n-s}{n\log n}+o(\frac{1}{n\log n})\Big)^{n-\log n+s}.
\end{align*}
This gives
\begin{align}\label{equ-aj-3}
\log^j j\sim (\log n)^{n-\log n+s}
\end{align}
since
\begin{align*}
\Big(1- \frac{\log n-s}{n\log n}+o(\frac{1}{n\log n})\Big)^{n-\log n+s}&=e^{(n-\log n+s)\log\Big(1- \frac{\log n-s}{n\log n}+o(\frac{1}{n\log n})\Big)}\\
&=e^{(n-\log n+s)\Big(- \frac{\log n-s}{n\log n}+o(\frac{1}{n\log n})\Big)}\sim 1.
\end{align*}
It follows from \eqref{equ-aj-1}-\eqref{equ-aj-3} that
$$
a_j\lesssim \frac{1}{( \log n )^{\log n-s+\frac{1}{2}}e^{-\log n+\frac{3}{2}\frac{s^2}{\log n}}(\log n)^{n-\log n+s}}= \frac{n}{(\log n)^{n+\frac{1}{2}}e^{\frac{3}{2}s^2/\log n}}.
$$
This gives the desired estimate.
\end{proof}

\begin{lemma}\label{lem-aj-bound-2}
If $j=n-L\log n$ with some constant $L>1$, then for $n$ sufficiently large
$$
a_j\lesssim_L \frac{1}{n^{L\log \frac{L}{e}} (\log n)^{n+\frac{1}{2}}}.
$$
\end{lemma}
\begin{proof}
The strategy is the same as above. First note that
$$
a_j \leq \frac{1}{(L\log n)!\log^jj}.
$$
Using the Stirling formula, we have
$$
(L\log n)!\sim \sqrt{L\log n}(\frac{L\log n}{e})^{L\log n}=\sqrt{L}n^{L\log \frac{L}{e}}(\log n)^{L\log n+\frac{1}{2}}.
$$
Moreover, we have $\log^jj=(\log(n-L\log n))^{n-L\log n}=e^{(n-L\log n)\log\log(n-L\log n)}$ and
\begin{align*}
&\log\log(n-L\log n)=\log (\log n+\log(1-\frac{L\log n}{n}))\\
&= \log (\log n-\frac{L\log n}{n} +o(\frac{L\log n}{n}) )\\
&=\log\log n+ \log(1-\frac{L}{n}+o(\frac{L}{n}))=\log\log n-\frac{L}{n}+o(\frac{L}{n}).
\end{align*}
This implies that
$$
\log^jj=e^{(n-L\log n)\Big(\log\log n-\frac{L}{n}+o(\frac{L}{n}) \Big)}=(\log n)^{n-L\log n}e^{(n-L\log n)\Big( -\frac{L}{n}+o(\frac{L}{n}) \Big)}\gtrsim_L (\log n)^{n-L\log n}.
$$
Combining the two bounds, we obtain the desired estimate.
\end{proof}

In the sequel, we shall use frequently the bound
\begin{align}\label{equ-aj-5}
(\log (n+10))^{n}\lesssim (\log n)^{n}.
\end{align}
Indeed,
\begin{align*}
(\frac{\log(n+10)}{\log n})^n=(1+\frac{\log(1+\frac{10}{n})}{\log n})^n=(1+\frac{1}{\log n}(\frac{10}{n}+o(\frac{10}{n})))^n\sim 1.
\end{align*}

The following lemma is the key step in our proof of the main results.
\begin{lemma}\label{lem-bootstrap}
Let $\kappa>1$ and $C\geq \kappa C_0$. Then  we have
$$
\sum_{l+j=n}\frac{C_0^lC^j}{l!\log^{j}(j+e)}\lesssim  (\frac{1}{\log \kappa}+1) \frac{(n+1) C^{n}}{\log^{n+1}(n+1+e)}, \quad \forall n\in \N.
$$
\end{lemma}
\begin{proof}
We first consider the case $n\gg 1$ such that Lemma \ref{lem-mono} can be applied.
Let $a_j$ be given by \eqref{equ-s3-1}. We split the sum as
\begin{align*}
\sum_{j+l= n} a_jC_0^lC^j&=  \sum_{j< n/2}+\sum_{\frac{n}{2}\leq j< n-3\log n}+\sum_{n-3\log n\leq j<n-\log n-(\log n)^{\frac{2}{3}}}\\
&\quad +\sum_{n-\log n-(\log n)^{\frac{2}{3}} \leq j\leq n-\log n-(\log n)^{\frac{2}{3}}}+\sum_{n-\log n+(\log n)^{\frac{2}{3}}<j\leq n}.
\end{align*}

{\bf (1) The case $j< n/2$.} Note that
\begin{align*}
\sum_{j<\frac{n}{2}}a_jC_0^lC^j=\sum_{j<\frac{n}{2}}\frac{C_0^lC^j}{l!\log^j (j+e)}\leq \sum_{j<\frac{n}{2}}\frac{C_0^lC^j}{\frac{n}{2}!}<n\frac{C^n}{\frac{n}{2}!}\lesssim \frac{nC^n}{\log^{n+1}(n+1+e)}
\end{align*}
since by the Stirling formula $n!\sim \sqrt{n}(n/e)^n$, when $n\to \infty$,
\begin{align*}
\frac{\log^{n+1} (n+1+e)}{\frac{n}{2}!}\lesssim \frac{e^{(n+1)\log\log(n+1+e)}}{\sqrt{\frac{n}{2}}\left(\frac{n}{2e}\right)^{n/2}}\lesssim 1.
\end{align*}

{\bf (2) The case $n/2\leq j< n-3\log n$.} In this case, by Lemma \ref{lem-mono} we see $a_j\leq a_{n-L\log n}$ with $L=3$. Thus, by Lemma \ref{lem-aj-bound-2} and noting $C_0<C$, we find
\begin{align*}
\sum_{n/2\leq j<n-3\log n}a_jC_0^lC^j&\lesssim \sum_{n/2\leq j<n-3\log n } C^na_{n-3\log n}\\
&\lesssim \sharp\{n/2\leq j<n-3\log n\}\frac{C^n}{n^{3\log \frac{3}{e}} (\log n)^{n+\frac{1}{2}}}\\
&\lesssim \frac{nC^n}{n^{3\log \frac{3}{e}}  \log^{n+\frac{1}{2}}( n+1+e)}  \lesssim \frac{nC^n}{\log^{n+1}(n+1+e)}.
\end{align*}
Here in the first inequality of last line, we used \eqref{equ-aj-5}.

{\bf (3) The case $n-3\log n\leq j<n-\log n-(\log n)^{\frac{2}{3}}$.} In this case, the idea is the same as Case (2). In fact, by Lemma \ref{lem-mono} we see $a_j\leq a_{n-\log n+s}$ with $s=-(\log n)^{\frac{2}{3}}$. Then by Lemma \ref{lem-aj-bound-1}, we infer that
\begin{align*}
&\sum_{n-3\log n\leq j<n-\log n-(\log n)^{\frac{2}{3}}}a_jC_0^lC^j\lesssim \sum_{n-3\log n\leq j<n-\log n-(\log n)^{\frac{2}{3}}} C^n \frac{n}{ (\log n)^{n+\frac{1}{2}}e^{\frac{3}{2}s^2/\log n}}\\
&\leq \sharp \{n-3\log n\leq j<n-\log n-(\log n)^{\frac{2}{3}}\} \frac{n C^n}{(\log n)^{n+\frac{1}{2}}e^{\frac{3}{2}(\log n)^{\frac{1}{3}}}}\\
&\lesssim \frac{nC^n\log n}{ \log^{n+\frac{1}{2}}( n+1+e)e^{\frac{3}{2}(\log n)^{\frac{1}{3}}}} \lesssim \frac{nC^n}{\log^{n+1}(n+1+e)}.
\end{align*}

{\bf (4) The case $n-\log n-(\log n)^{\frac{2}{3}} \leq j\leq n-\log n+(\log n)^{\frac{2}{3}}$.}
By Lemma \ref{lem-mono} we have $a_j\lesssim a_{n-\log n+s}$ with $s=O(1)$. Then by Lemma \ref{lem-aj-bound-1}, we find $a_j\lesssim \frac{n}{ (\log n)^{n+\frac{1}{2}} }$. Moreover, we exploit the fact that $C\geq \kappa C_0$ and deduce that
\begin{align*}
\sum_{n-\log n-(\log n)^{\frac{2}{3}}\leq j<n-\log n+(\log n)^{\frac{2}{3}}}a_jC_0^lC^j&\lesssim \sum_{n-\log n-(\log n)^{\frac{2}{3}}\leq j<n-\log n+(\log n)^{\frac{2}{3}}} \kappa^{-l}C^n \frac{n}{ (\log n)^{n+\frac{1}{2}}}\\
&\lesssim (\log n)^{\frac{2}{3}} \kappa^{-(\log n-(\log n)^{\frac{2}{3}})} \frac{n C^n}{(\log n)^{n+\frac{1}{2}}}\\
&\lesssim (\log n)^{\frac{2}{3}+\frac{1}{2}} \kappa^{-(\log n-(\log n)^{\frac{2}{3}})} \frac{n C^n}{(\log n)^{n+1}}\\
&\lesssim   (\frac{1}{\log \kappa}+1) \frac{nC^n}{\log^{n+1}(n+1+e)},
\end{align*}
since for $n$ sufficiently large, noting $\kappa>1$,
\begin{align*}
(\log n)^{\frac{2}{3}+\frac{1}{2}} \kappa^{-(\log n-(\log n)^{\frac{2}{3}})}\lesssim  (\log n) \kappa^{-\frac{1}{2}\log n}
\left\{
\begin{array}{ll}
 \lesssim \frac{1}{\log \kappa}, & \mbox{ if } \kappa <e,\\
 \lesssim 1, & \mbox{ if } \kappa \geq e.
\end{array}
\right.
\end{align*}

{\bf (5) The case $n-\log n+(\log n)^{\frac{2}{3}}<j\leq n$.} In this case, by Lemmas \ref{lem-aj-bound-1} and \ref{lem-aj-bound-2}, we find $a_j\leq a_{n-\log n+(\log n)^{\frac{2}{3}}}\lesssim \frac{n }{(\log n)^{n+\frac{1}{2}}e^{\frac{3}{2}(\log n)^{\frac{1}{3}}}}$. Thus,
\begin{align*}
\sum_{n-\log n+(\log n)^{\frac{2}{3}}\leq n}a_jC_0^lC^j \lesssim \log n \frac{nC^n }{(\log n)^{n+\frac{1}{2}}e^{\frac{3}{2}(\log n)^{\frac{1}{3}}}} \lesssim \frac{nC^n}{\log^{n+1}(n+1+e)}.
\end{align*}

Finally, combining the five cases together, we obtain
$$
\sum_{l+j=n}\frac{C_0^lC^j}{l!\log^{j}(j+e)}\lesssim (\frac{1}{\log \kappa}+1)\frac{(n+1)C^n}{\log^{n+1}(n+1+e)}
$$
for $n$ sufficiently large, say $n> n_0$, where $n_0$ is a numerical constant.

If $n\leq n_0$,   setting $C'=\log^{n_0+1}(n_0+1+e)$, then
$$
\sum_{l+j=n}\frac{C_0^lC^j}{l!\log^{j}(j+e)}\leq \sum_{l+j=n}C^n=(n+1)C^n\leq C' \frac{(n+1)C^n}{\log^{n+1}(n+1+e)}.
$$
This completes the proof.
\end{proof}
%

\begin{proof}[Proof of Theorem \ref{thm-1}]
%
%

Let $p\in[1,\infty]$ and $\kappa>1$. It follows from Lemma \ref{lem-bootstrap}  that
\begin{align}\label{equ-boot-d1-810}
\sum_{l+j=n}\frac{C_0^lC^j}{l!\log^{j}(j+e)}\leq K_1(\frac{1}{\log \kappa}+1) \frac{(n+1) C^{n}}{\log^{n+1}(n+1+e)}, \quad \forall n\in \N,
\end{align}
where $K_1>0$ is a numerical constant.  We can assume $K_1\geq 1$.

Let $u$ be a solution of \eqref{equ-elip} with $\|u\|_{L^p(\R)}=1$. We claim that there exist a constant $K_2=K_2(p)>0$ depending only on $p$ such that
\begin{align}\label{equ-proof-d1-3}
 \|\partial_x^{k}u\|_{L^p(\R)}\leq \frac{C^kk!}{\log^k(k+e)}, \quad \forall k\in \N,
\end{align}
where $C=\kappa C_0+(7K_1+K_2)(\frac{1}{\log \kappa}+1)>C_0+1$.

We prove \eqref{equ-proof-d1-3} by induction. Assuming that \eqref{equ-proof-d1-3} holds for all $k\leq n$ with some $n\geq 2$, we consider the case $k=n+1$. Using the equation $\partial_x^2u=W(x)\partial_xu+V(x)u$, we have
\begin{align*}
&\|\partial_x^{n+1}u\|_{L^p(\R)}=\|\partial_x^{n-1}(W(x)\partial_xu+V(x)u)\|_{L^p(\R)}\\
&\leq \sum_{l+j=n-1}\frac{(n-1)!}{l!j!}\|\partial_x^lW\|_{L^\infty(\R)}\|\partial_x^{j+1}u\|_{L^p(\R)} + \sum_{l+j=n-1}\frac{(n-1)!}{l!j!}\|\partial_x^lV\|_{L^\infty(\R)}\|\partial_x^{j}u\|_{L^p(\R)}\\
&\leq \sum_{l+j=n}\frac{n!}{l!j!}\|\partial_x^lW\|_{L^\infty(\R)}\|\partial_x^{j}u\|_{L^p(\R)} + \sum_{l+j=n-1}\frac{(n-1)!}{l!j!}\|\partial_x^lV\|_{L^\infty(\R)}\|\partial_x^{j}u\|_{L^p(\R)}={\rm I}+{\rm II}.
\end{align*}
Using the induction hypothesis \eqref{equ-proof-d1-3} with $k\leq n$, we find
\begin{align*}
{\rm I}&\leq \sum_{l+j=n}\frac{n!}{l!j!}C_0^l\frac{ C^jj!}{\log^j(j+e)}=\sum_{l+j=n} C_0^l C^j \frac{n!}{l!\log^j(j+e) }\leq  K_1(\frac{1}{\log \kappa}+1) \frac{(n+1)!  C^{n}}{\log^{n+1}(n+1+e)}
\end{align*}
by \eqref{equ-boot-d1-810}. Similarly, noting $C>1$,
\begin{align*}
{\rm II}&\leq \sum_{l+j=n-1} C_0^l C^j \frac{(n-1)!}{l!\log^j(j+e) } \leq   K_1(\frac{1}{\log \kappa}+1)\frac{n! C^{n-1}}{\log^{n}(n+e)} \\
&\leq  K_1(\frac{1}{\log \kappa}+1)\frac{ (n+1)!C^{n}}{\log^{n+1}(n+1+e)}d_n,
\end{align*}
where we used $d_n=\frac{\log(n+1+e)}{n+1}(\frac{\log(n+1+e)}{\log(n+e)})^{n}<6$ for all $n\geq 2$. In fact,
\begin{align*}
(\frac{\log(n+1+e)}{\log(n+e)})^{n}&=(1+\frac{1}{\log(n+e)}\log(1+\frac{1}{n+e}))^{n}\leq (1+\frac{1}{\log(n+e)} \frac{1}{n+e} )^{n}\\
&=e^{n\log(1+\frac{1}{\log(n+e)} \frac{1}{n+e} )}\leq e^{n \frac{1}{\log(n+e)} \frac{1}{n+e}  }<3
\end{align*}
and $\frac{\log(n+1+e)}{n+1}<2$ for all $n\geq2$.

Combining the bounds of ${\rm I}$ and ${\rm II}$, we obtain
$$
 \|\partial_x^{n+1}u\|_{L^p(\R)}\leq 7K_1(\frac{1}{\log \kappa}+1)\frac{ (n+1)!C^{n}}{\log^{n+1}(n+1+e)}.
$$
Thus, \eqref{equ-proof-d1-3} is valid with $k=n+1$.

To complete the proof, we show that \eqref{equ-proof-d1-3} holds with $k\leq 2$. In fact, since $\|u\|_{L^p(\R)}=1$ and $\|W\|_{L^\infty(\R)}, \|V\|_{L^\infty(\R)}\leq 1$, from the equation \eqref{equ-elip} we have
\begin{align}\label{equ-1d-89-1}
\|\partial_x^2u\|_{L^p(\R)}\leq \|\partial_x u
\|_{L^p(\R)}+1.
\end{align}
Moreover, there exists a constant $c_p\geq 1$ depending only on $p$ such that
\begin{align}\label{equ-1d-89-2}
\|\partial_xu\|_{L^p(\R)}\leq c_p\|\partial^2_xu\|^{1/2}_{L^p(\R)}\|u\|^{1/2}_{L^p(\R)}=c_p\|\partial^2_xu\|^{1/2}_{L^p(\R)};
\end{align}
see the interpolation inequality in \cite{AdamsSobolev} for $1\leq p<\infty$ and the Landau inequality if $p=\infty$ (one can take  $c_p=2$ in this case). It follows from \eqref{equ-1d-89-1}-\eqref{equ-1d-89-2} that
\begin{align}\label{equ-1d-89-3}
\|\partial_x^2u\|_{L^p(\R)}\leq  c_p \|\partial_x^2u\|^{1/2}_{L^p(\R)}+1.
\end{align}
Since 
$$
c_p\|\partial_x^2u\|^{1/2}_{L^p(\R)}\leq \frac{1}{2} \|\partial_x^2u\|_{L^p(\R)}+\frac{1}{2}c_p^2,
$$ 
we deduce from \eqref{equ-1d-89-3} that
$$
\|\partial_x^2u\|_{L^p(\R)} \leq 2+c_p^2.
$$
By \eqref{equ-1d-89-2} again, we have $\|\partial_xu\|_{L^p(\R)}\leq  c_p (2+c_p^2)^{1/2}$. Thus, \eqref{equ-proof-d1-3} holds for $k\leq 2$ with $K_2=2(2+c_p^2)$.

Finally, letting $K=7K_1+K_2$, we obtain \eqref{equ-ultra-ana}. This ends the proof.    
\end{proof}


Similarly to the approach used in the proof of Theorem \ref{thm-1}, we can demonstrate the following result.
\begin{corollary}\label{cor-95-1}
Let $p\in[1,\infty]$ and $u\in L^p(\R)$, with $\|u\|_{L^p(\R)}=1$, be a solution to $\partial_x u =W(x)u$ with $W$ satisfying
$$
\|\partial_x^jW\|_{L^\infty(\R)}\leq C_0^j, \quad \forall j\in \N.
$$
Then there exists a constant $K=K(p)>0$ depending only on $p$ such that, for any $\kappa>1$,
the solution satisfies the log-type ultra-analytic upper bound
\begin{align}\label{equ-ultra-ana-coro}
\|\partial_x^n u\|_{L^p(\R)}\leq \Big( \kappa C_0+K(\frac{1}{\log \kappa}+1)\Big)^n\frac{ n!}{\log^{n}(n+e)}, \quad \forall n\in \N.
\end{align}
\end{corollary}


We note that similar results apply to the normalized \( L^2(\mathbb{R}) \) solutions of \( |D| u = W(x) u \). Specifically, by Plancherel's theorem, we have \( \|\partial_x^{n+1} u\|_{L^2(\mathbb{R})} = \|\partial_x^n |D| u\|_{L^2(\mathbb{R})} \) for all \( n \in \mathbb{N}_0 \). Consequently, the same induction argument yields the ultra-analytic upper bound \eqref{equ-ultra-ana-coro} with \( p = 2 \). This highlights a fundamental difference between the ultra-analyticity of the elliptic equation \( |D| u = W u \) and the corresponding parabolic equation \( \partial_t u + |D| u = W u \). As mentioned in the introduction, the solution of the latter cannot be ultra-analytic for any \( t > 0 \).

\subsection{Sharpness of  Theorem \ref{thm-1}}

In this subsection, we provide an example demonstrating that Theorem \ref{thm-1} is sharp in two respects. The main result is presented in the following proposition. It shows that the ultra-analytic estimate \eqref{equ-ultra-ana} cannot be improved to
$$
\sup_{x \in \mathbb{R}^d} |\partial_x^\alpha u(x)| \leq C^\alpha (\alpha!)^{1 - \varepsilon}, \quad \forall \alpha \in \mathbb{N}_0^d,
$$
for any \( \varepsilon > 0 \). Specifically, this indicates that generic solutions to the elliptic equation \eqref{equ-elip} with coefficients of exponential type exhibit weaker ultra-analytic properties compared to the eigenfunctions of the harmonic oscillator.

\begin{proposition}\label{prop-sharp}
There is a solution $u\in L^\infty(\R)$ satisfying $ \|u\|_{L^\infty(\R)}=1$ of $\partial_x^2u=W(x)\partial_xu+V(x)u$ with $W,V$ satisfying
\begin{align}\label{equ-sharp-1}
\sup_{x\in \R}|\partial_x^n W(x)|+\sup_{x\in \R}|\partial_x^n V(x)|\leq C_0^n, \quad \forall n\in \N
\end{align}
for some constant $C_0>0$, such that \eqref{equ-ultra-ana} holds with $p=\infty$ and the following statements hold true.
\begin{itemize}
  \item [(1)] For every $C>0, \lambda>1$, the solution $u$ does not satisfy
\begin{align}\label{equ-sharp-2}
\|\partial_x^n u\|_{L^\infty(\R)} \leq \frac{C^nn!}{\log^{\lambda n}(n+e)}, \quad \forall n\in \N.
\end{align}
  \item [(2)] For every $0<\kappa<1, C>0$, the  solution $u$ does not satisfy
\begin{align}\label{equ-sharp-3}
\|\partial_x^n u\|_{L^\infty(\R)} \leq (\kappa C_0+C)^n\frac{ n!}{\log^{n}(n+e)}, \quad \forall \alpha\in \N
\end{align}
 when $C_0$ is sufficiently large.
\end{itemize}
\end{proposition}

\begin{proof}
For every $C_0>0$, let
$$
u(x)=e^{-A}e^{\phi(x)}, \quad \phi(x)=Ae^{iC_0x}, x\in \R
$$
and $A=\frac{1}{1+C_0^2}$.
It is easy to check that $\|u\|_{L^\infty(\R)}=1$ and $\partial_x^2u=W(x)\partial_xu+V(x)u$ holds with $W=\phi'$ and $V=\phi''$. Since $\|\partial_x^n\phi\|_{L^\infty(\R)}\leq AC_0^{n}$, we find
$$
\|\partial_x^n W\|_{L^\infty(\R)},\, \|\partial_x^n V\|_{L^\infty(\R)}\leq C_0^{n}
$$
for all $n\in N$ due to our choice of $A$.

Next, we check that the function $u$ satisfies that for every $\kappa>1$, there exists a constant $K$ such that
\begin{align}\label{equ-sharp-4}
\|\partial_x^n u\|_{L^\infty(\R)}\leq  \Big( \kappa C_0+K(\frac{1}{\log \kappa}+1)\Big)^n\frac{ n!}{\log^{n}(n+e)}, \quad \forall n\in \N.
\end{align}
Note  that \eqref{equ-sharp-4} follows directly from Theorem \ref{thm-1}. Nevertheless, we shall give a direct proof here, which is of independent interest. Indeed, it is easy to check that 
$$
\|u'\|_{L^\infty(\R)}\leq \|\phi'\|_{L^\infty{(\R)}}\leq 1\quad 
\text{and}\quad 
\|u''\|_{L^\infty(\R)}\leq \|\phi''+(\phi')^2\|_{L^\infty(\R)}\leq 2,
$$ 
thus \eqref{equ-sharp-4} holds for $n\leq 2$. Now we assume $n\geq 3$.  Extend $u(x)$ to an analytic function on the complex plane $\C$, namely
$u(z)=e^{-A}e^{Ae^{iC_0z}}, z\in \C.$
By the Cauchy formula,
$$
u^{(n)}(x)=\frac{n!}{2\pi i}\int_{\mathcal {C}} \frac{u(z)}{(z-x)^{n+1}}\d z, \quad x\in \R,
$$
where $\mathcal {C}$ is the circle $\{z\in\C: |z-x|=r\}$ and $r>0$ is to be determined later. Writing $z=x+re^{i\theta}$, then
\begin{align*}
|u^{(n)}(x)|\leq \frac{n!}{2\pi}\left|\int_{\mathcal {C}} \frac{u(z)}{(z-x)^{n+1}}\d z\right|\leq \frac{n!}{2\pi} \int_{0}^{2\pi} \frac{|u(x+re^{i\theta})|}{r^{n}}\d \theta \leq n!r^{-n}e^{-A}e^{Ae^{C_0r}}
\end{align*}
for all $x\in \R$. Choosing $r=C_0^{-1}\log\frac{n}{A\log n}$, we find
\begin{align*}
 \|u^{(n)}\|_{L^\infty(\R)}\leq n!(C_0^{-1}\log\frac{n}{A\log n})^{-n}e^{-A}e^{\frac{n}{\log n}}\leq  n!C_0^n(\log n)^{-n} (1-\frac{\log (A\log n)}{\log n})^{-n}e^{\frac{n}{\log n}}.
\end{align*}
Thus, \eqref{equ-sharp-4} follows if we show that there exists $K>0$ such that
$$
   C_0 \frac{\log (n+e)}{\log n}  (1-\frac{\log (A\log n)}{\log n})^{-1}e^{\frac{1}{\log n}}\leq \kappa C_0+K(1+\frac{1}{\log \kappa})
$$
holds for all $C_0>0, \kappa>1$. Noting $\frac{\log (n+e)}{\log n}=1+O(\frac{1}{n})$ (see \eqref{equ-aj-5}), we find $\frac{\log (n+e)}{\log n}e^{\frac{1}{\log n}}\leq 1+\frac{K_1}{\log n}$ for some $K_1>0$. Recalling $A=\frac{1}{1+C_0^2}$, it suffices to show
\begin{align}\label{equ-95-1}
C_0 \frac{\log n+K_1}{\log n+\log(1+C_0^2)-\log\log n} \leq \kappa C_0+K(1+\frac{1}{\log \kappa}).
\end{align}
If $\log(1+C_0^2)\geq\log \log n+K_1$, then \eqref{equ-95-1} holds since the left-hand side is bounded by $C_0$. Thus, we may assume that
\begin{align}\label{equ-95-1.5}
  \log(1+C_0^2)< \log \log n+K_1.
\end{align}
In this case, \eqref{equ-95-1} can be rewritten as the following equivalent form:
\begin{align}\label{equ-95-2}
C_0 \frac{\log\log n+K_1-\log(1+C_0^2)}{\log n+\log(1+C_0^2)-\log\log n} \leq (\kappa -1)C_0+K(1+\frac{1}{\log \kappa}).
\end{align}
We assume in the sequel
\begin{align}\label{equ-95-3}
\frac{\log\log n+K_1-\log(1+C_0^2)}{\log n+\log(1+C_0^2)-\log\log n}\geq \kappa-1
\end{align}
since otherwise \eqref{equ-95-2} holds clearly. From \eqref{equ-95-3}, we have
\begin{align*}
\log n&\leq \log\log n-\log(1+C_0^2)+\frac{1}{\kappa-1}(\log\log n+K_1-\log(1+C_0^2))\\
&\leq \log\log n +\frac{1}{\kappa-1}(\log\log n+K_1)\leq (1+\frac{1}{\kappa-1})(\log\log n+K_1).
\end{align*}
It follows that for some numerical constant $K_2>0$,
\begin{align}\label{equ-95-4}
1+\frac{1}{\kappa-1}\geq \frac{\log n}{\log\log n+K_1}\geq K_2(\log n)^{\frac{1}{2}}.
\end{align}
Thus, the left-hand side of \eqref{equ-95-2} is bounded by
\begin{align*}
 C_0 \frac{\log\log n+K_1}{\log n-\log\log n} &\lesssim C_0 \stackrel{\mbox{ \footnotesize by}\eqref{equ-95-1.5}}{\leq} (e^{K_1}\log n)^{\frac{1}{2}}\stackrel{\mbox{ \footnotesize by}\eqref{equ-95-4}}{\leq}  e^{K_1/2}K_2^{-1}(1+\frac{1}{\kappa-1} )\\
 &\lesssim e^{K_1/2}K_2^{-1}(1+\frac{1}{\log \kappa} )   
\end{align*}
as desired.

\medskip

We now show that the function \( u \) does not satisfy the estimates \eqref{equ-sharp-2} or \eqref{equ-sharp-3}. The approach involves determining the growth rates of \( |u(z)| \) as \( z \to \infty \) under these bounds and then showing that these rates lead to a contradiction to the definition of \( u(x) \).

(1) Suppose that \eqref{equ-sharp-2} holds.
Then for every $z\in \C$, by the Taylor expansion and using  \eqref{equ-sharp-2} with $x=0$, we find
$$
|u(z)|\leq \sum_{n}|\frac{u^{(n)}(0)}{n!}||z|^n \leq C(1+r)+\sum_{n>1} \frac{C^n }{\log^{\lambda n}n}r^n, \quad r=|z|.
$$
Split the sum into two terms $S_1+S_2$, where $S_1$ contains the terms with $n$ satisfying $Cr<(\log n)^\lambda n^{-\frac{2}{n}}$ and $S_2$ contains the remaining terms. For $S_1$, we have
$$
S_1\leq \sum_{n>1,Cr<(\log n)^\lambda n^{-\frac{2}{n}}} n^{-2}<\frac{\pi^2}{6}.
$$
For $S_2$, we have $Cr\geq (\log n)^\lambda n^{-\frac{2}{n}}$, which implies that
$$
\log n\leq (Cr)^{\frac{1}{\lambda}}n^{\frac{2}{\lambda n}}\leq (Cr)^{\frac{1}{\lambda}}e^{\frac{2}{\lambda }}=(Ce^2r)^{\frac{1}{\lambda}},
$$
where we used $n^{1/n}\leq e$. Thus, $n\leq e^{(Ce^2r)^{\frac{1}{\lambda}}}:=N$, and we bound $S_2$ by
$$
S_2 \leq \sum_{n>1} \frac{C^n }{\log^{\lambda n}n}r^N\leq C'r^N=C'r^{e^{(Ce^2r)^{\frac{1}{\lambda}}}}=C'e^{(\log r)e^{(Ce^2r)^{\frac{1}{\lambda}}}}<C'e^{ e^{(2Ce^2r)^{\frac{1}{\lambda}}}}
$$
when $r$ is large.
Combining the above bounds, we see that for some large constant $C''>0$,
$$
|u(z)|\leq C''(1+|z|)+C''e^{e^{(C''|z|)^{\frac{1}{\lambda}}}}.
$$
In particular, letting $z=iy$, we see
$$
|u(iy)|\leq C''(1+|y|)+C''e^{e^{(C''|y|)^{\frac{1}{\lambda}}}}, \quad \forall y\in \R.
$$
If $\lambda>1$, the growth rate leads to a contradiction since $|u(iy)|= e^{-A} e^{Ae^{-C_0y}}\sim  e^{Ae^{C_0|y|}}$ as $y\to -\infty$.

(2) The idea is similar. In fact, by the Taylor expansion, we have
$$
|u(z)|\leq  C(1+r)+\sum_{n>1} \frac{(\kappa C_0+C)^n }{\log^{n}n}r^n, \quad r=|z|.
$$
Split the sum as $S_1+S_2+S_3$, where $S_1$ contains the sum for all $n$ such that $\log n\geq n^{2/n}r(\kappa C_0+C)$, $S_2$ contains the sum for all $n>N_0$ such that $\log n< n^{2/n}r(\kappa C_0+C)$, and $S_3$ contains the remaining terms, where $N_0$ is to be determined. Then we have $S_1\lesssim 1$ as before. Since $\lim_{n\to \infty} n^{2/n}=1$, we have for any $\varepsilon>0$, 
$$ 
n^{2/n}r(\kappa C_0+C)<(1+\varepsilon)r(\kappa C_0+C)
$$ 
if $n>N_0$ and $N_0$ is sufficiently large. Thus, with such $N_0$,  $S_2$ only contains the sum for $n\leq e^{(1+\varepsilon)r(\kappa C_0+C)}$. Hence
$$
S_2\lesssim \Big(r(\kappa C_0+C) \Big)^{e^{(1+\varepsilon)r(\kappa C_0+C)}}
=e^{e^{(1+\varepsilon)r(\kappa C_0+C)}\log\big(r(\kappa C_0+C) \big)}\leq  e^{C_\varepsilon e^{(1+2\varepsilon)r(\kappa C_0+C)}}.
$$
If $\kappa<1$, then we have $S_2\lesssim e^{C_\varepsilon e^{(1-\varepsilon)C_0 r + 3Cr}}$ for $\varepsilon$ is small enough. Moreover, it is easily seen that $S_3\lesssim (r(\kappa C_0+C))^{N_0}$. Summing up, we find
$$
|u(z)|\lesssim 1+|z|+|z|^{N_0}(\kappa C_0+C)^{N_0}+e^{C_\varepsilon e^{(1-\varepsilon)C_0 |z| + 3C|z|}}, \quad \forall z\in \C.
$$
It follows that $|u(iy)|\lesssim e^{C_\varepsilon e^{(1-\varepsilon)C_0 |y| + 3C|y|}}$.
This leads to a contradiction when $C_0\to \infty$ since $|u(iy)|\sim e^{Ae^{C_0|y|}}$ as $y\to -\infty$.
\end{proof}

Theorem \ref{thm-1} provides log-type ultra-analytic bounds for elliptic equations with lower-order terms and variable coefficients. A natural extension of this question is to examine what happens with second-order elliptic equations that have variable leading coefficients. Consider the elliptic equation
\begin{align}\label{equ-div-830-1}
  \partial_x(A(x)\partial_x u) = W(x)\partial_x u + V(x)u, \quad x \in \R,
\end{align}
where the function \(A: \R \to \C\) is non-trivial (i.e., \(A \neq 0\)) and satisfies
\begin{align}\label{equ-div-830-2}
 \|\partial_x^n A\|_{L^\infty(\R)} \leq C_1^n, \quad \forall n \in \N,
\end{align}
for some constant \(C_1 > 0\). The following result shows that similar ultra-analytic bounds to those in Theorem \ref{thm-1} hold for the solutions of \eqref{equ-div-830-1} if and only if \(A\) equals to a multiple of \(e^{iCx}\) for some constant \(C\). This implies that the ultra-analytic bounds only hold in this trivial case, which reduces to the elliptic equation with constant leading coefficients studied in Theorem \ref{thm-1}. Specifically, when \(A\) takes the form \(e^{iCx}\), the equation \eqref{equ-div-830-1} can be rewritten as a constant coefficient case \eqref{equ-elip} with modified coefficients \(W\) and \(V\). The proof below indicates that the divergence form of \eqref{equ-div-830-1} is not crucial; the same conclusion applies to the non-divergence form.

\begin{proposition}\label{prop-div}
Assuming \eqref{equ-div-830-2}, the following statements are equivalent.
\begin{itemize}
    \item [(i)] Assume that $\|\partial_x^n W\|_{L^\infty(\R)}, \|\partial_x^nV\|_{L^\infty(\R)}\leq C_0^{n+1}$ for all $n\in \N$ with some constant $C_0>0$, then there exists a constant $C>0$ such that the solutions of \eqref{equ-div-830-1} satisfy
\begin{align}\label{equ-div-830-3}
    \|\partial_x^n u\|_{L^\infty(\R)}\leq \frac{C^{n+1}n!}{\log^n(n+e)}, \quad \forall n\in \N.
\end{align}
    \item [(ii)] There exist constants $C_2\in \R$ satisfying $|C_2|\leq C_1$, and $C_3\in \C$ satisfying $\emph{Re } C_3\leq 0$ such that $A(x)=e^{iC_2x+C_3}, x\in \R$.
\end{itemize}
\end{proposition}
\begin{proof}
It is convenient to rewrite \eqref{equ-div-830-1} into a non-divergence form as
\begin{align}\label{equ-div-830-4}
  A(x)\partial_x^2 u = \widetilde{W(x)}\partial_xu+V(x)u, \quad x\in \R,
\end{align}
where $\widetilde{W}=W-\partial_x A$.

 $(ii)\Longrightarrow(i)$. If $A(x)=e^{iC_2x+C_3}$, where $C_2\in \R$ satisfies $|C_2|\leq C_1$, and $C_3\in \C$ satisfies $\mbox{Re } C_3\leq 0$, then the equation \eqref{equ-div-830-4} reduces to the one-dimensional case of \eqref{equ-elip}. Indeed, we can rewrite further \eqref{equ-div-830-4} as
 \begin{align}\label{equ-div-830-5}
\partial_x^2 u = W_1\partial_xu+V_1u, \quad x\in \R,
\end{align}
where $W_1=WA^{-1}-A^{-1}\partial_xA$ and $V_1=VA^{-1}$. It is easy to check  that
$$
\|\partial_x^n W_1\|_{L^\infty(\R)}\leq C_0(C_0+C_1)^n+C_1, \|\partial_x^n V_1\|_{L^\infty(\R)}\leq C_0(C_1+C_0)^n\quad \forall n\in \N.
$$
 Then Theorem \ref{thm-1} gives the bound \eqref{equ-div-830-3}. Note that the factor $C^{n+1}$ (not $C^n$) is necessary since $W_1$ is not bounded by $1$ and we do not assume that $\|u\|_{L^\infty(\R)}=1$.

$(i)\Longrightarrow(ii)$. We first note that under the assumption \eqref{equ-div-830-2}, $A$ can be extended to an entire function on $\C$, still denoted by $A$. Now we claim that if $(i)$ holds, then the map $A:\C\mapsto \C$ has no zeros. In other words,
\begin{align}\label{equ-div-830-6}
    A(z)\neq 0, \quad \forall z\in \C.
\end{align}
Once this claim is proved, we then conclude that $A(x)=e^{iC_2x+C_3}$ for some $C_2\in \R$ with $|C_2|\leq C_1$, $C_3\in \C$ with $\mbox{Re }C_3\leq 0$. Indeed, since $A$ is an entire function and have no zeros, then there exists an entire function $g$ such that
\begin{align}\label{equ-div-830-7}
    A(z)=e^{g(z)}, \quad \forall z\in \C;
\end{align}
see e.g.  \cite[Theorem 6.2]{Steincomplex}. Moreover, by the assumption \eqref{equ-div-830-2}, $A$ is of the exponential type, namely,
\begin{align}\label{equ-div-830-8}
 |A(z)|\leq e^{C_1|z|}, \quad z\in \C.
\end{align}
 Thus, we infer from \eqref{equ-div-830-7}-\eqref{equ-div-830-8} that $|g(z)|\lesssim C_1|z|$ as $|z|\to \infty$. By  \cite[Lemma 5.5]{Steincomplex}, $g$ is a polynomial of degree $\leq 1$. In other words,
$$
g(z)=az+b, \quad z\in \C
$$
for some constants $a,b\in \C$. In particular, $A(x)=e^{ax+b}, x\in \R$. Since $|A(x)|\leq 1$ for all $x\in \R$, we have $\mbox{ Re } a=0$ and $e^{\mbox{Re }b}\leq 1$ (or equivalently $\mbox{Re }b\leq 0$). The derivative bounds of $A$ show that $|\mbox{Im }a|\leq C_1$. This proves the desired conclusion.

It remains to prove the claim \eqref{equ-div-830-6}. We shall use a contradiction argument. Suppose that \eqref{equ-div-830-6} fails. Then there exists a $z_0\in \C$ such that $A(z_0)=0$. Since $A$ is not the zero function, by the unique continuation of analytic functions, $z_0$ is a zero of finite order, say $1\leq k\in \N$. Then the Taylor expansion of the function $A$ in $z_0$ is given by
$$
A(z)=\sum_{n\in \N} \frac{A^{(n)}(z_0)}{n!}(z-z_0)^n=\sum_{n\geq k} \frac{A^{(n)}(z_0)}{n!} (z-z_0) n, \quad z\in \C.
$$
Then we can write
\begin{align}\label{equ-div-830-9}
    A(z)=(z-z_0)^kh(z), \quad z\in \C,
\end{align}
where
\begin{align}\label{equ-div-830-10}
h(z)=\sum_{n\geq k}\frac{A^{(n)}(z_0)}{n!}(z-z_0)^{n-k}, \quad z\in \C.
\end{align}
Next, we show that $h$ is an analytic function of exponential type. Indeed, by \eqref{equ-div-830-8} and the Cauchy formula,
$$
|A^{(n)}(z_0)|\leq \frac{n!}{2\pi}\left|\int_\mathcal{C} \frac{A(z)}{(z-z_0)^{n+1}}\d z\right|\leq \frac{n!}{2\pi}\int_0^{2\pi} \frac{|A(z_0+re^{i\theta})|}{r^n}\d \theta \leq n!r^{-n}e^{C_1(r+|z_0|)},
$$
where $\mathcal {C}$ is the circle $\{z\in\C: |z-z_0|=r\}$, $r>0$ is arbitrary. Setting $r=n/C_1$ and using the Stirling formula, we obtain $|A^{(n)}(z_0)|\lesssim \sqrt{n}C_1^ne^{C_1|z_0|}$. Thus, we infer from \eqref{equ-div-830-10} that
\begin{align}\label{equ-div-830-11}
   |h(z)|\lesssim \sum_{n\geq k} \frac{\sqrt{n}C_1^ne^{C_1|z_0|}}{n!}\leq e^{C_1|z_0|}\sum_{n\in \N} \frac{(2C_1)^n}{n!} \leq e^{C_1(|z_0|+2)}, \quad \mbox{ if } |z-z_0|<1.
\end{align}
To get a bound in the case $|z-z_0|\geq 1$, we observe that
\begin{align}\label{equ-div-830-12}
|A(x+iy)|\leq e^{C_1|y|}, \quad \forall x,y\in \R.
\end{align}
In fact, owing to \eqref{equ-div-830-8} and $|A(x)|\leq 1$ for all $x\in \R$, \eqref{equ-div-830-12} follows from the Phragm\'{e}n-Lindel\"{o}f theorem; see e.g. \cite[Theorem 11, p.70]{Younganintroduction}. Then we infer from \eqref{equ-div-830-9} and \eqref{equ-div-830-12} that
\begin{align}\label{equ-div-830-13}
|h(z)|\leq \frac{|A(z)|}{|z-z_0|^k}\leq e^{C_1|y|}, \quad \mbox{ if } z=x+iy,\, |z-z_0|\geq 1.
\end{align}
Combining \eqref{equ-div-830-11} and \eqref{equ-div-830-13}, we get
\begin{align}\label{equ-div-830-14}
|h(x+iy)|\leq e^{C_1(|z_0|+2)}e^{C_1|y|}, \quad \forall x,y\in \R.
\end{align}
With \eqref{equ-div-830-14} in hand, using the Cauchy formula again, one can show that
\begin{align}\label{equ-div-830-15}
\|\partial^n_x h\|_{L^\infty(\R)}\leq C^n, \quad \forall n\in \N
\end{align}
for some constant $C>0$ depending on $|z_0|$ and $C_1$.

To get a contradiction, we show that, with a well-chosen $W$ and $V$, there is a solution of \eqref{equ-div-830-1} which does not satisfy \eqref{equ-div-830-3}.
Indeed, in light of \eqref{equ-div-830-15}, letting $W=h+\partial_xA$ and $V=0$, then $\|\partial_x^n W\|_{L^\infty(\R)}, \|\partial_x^nV\|_{L^\infty(\R)}\leq C_0^{n+1}$ for all $n\in \N$ with some constant $C_0>0$. With such choice, the equation \eqref{equ-div-830-4} then becomes
\begin{align}\label{equ-div-830-16}
 (x-z_0)^k\partial_x^2u=\partial_x u, \quad x\in \R.
\end{align}
Solving \eqref{equ-div-830-16} shows that there exists a solution $u$ satisfying
$$
\partial_xu(x)=
\left\{
\begin{array}{cc}
  x-z_0   &  \mbox{ if } k=1 \\
  e^{\frac{1}{1-k}(x-z_0)^{1-k}}   &  \mbox{ if } k\geq 2.
\end{array}
\right.
$$
In the case when $k=1$, $\partial_x u$ is not uniformly bounded on $\R$, thus \eqref{equ-div-830-3} fails. In the case when $k\geq 2$, $\partial_xu$ is not analytic at the point $z=z_0$, which also leads to a contradiction to \eqref{equ-div-830-3}, since \eqref{equ-div-830-3} implies that $u$ is an entire function. This completes the proof.
\end{proof}

\section{Higher dimensional case: Proof of Theorem \ref{thm-2}}
In the sequel, we also use $D^\alpha$ to denote the multi-index partial derivative $\partial_x^\alpha, \alpha\in \N^d$. For some technical reasons, we split our discussion into three cases, $1<p<\infty$, $p=1$, and $p=\infty$.

\subsection{The case \texorpdfstring{$1< p<\infty$}{}}

We first extend Lemma \ref{lem-bootstrap} to higher dimensions.
\begin{lemma}\label{lem-bootstrap-high}
Let $\kappa>1$ and $C\geq \kappa C_0$. Then there exists a constant $C_\kappa>0$ depending only on $\kappa$ such that for all multi-index $\alpha\in \N^d$
\begin{align}\label{equ-high-bi-0}
\sum_{\beta,\gamma\in \N^d: \beta+\gamma=\alpha}\frac{\alpha!}{\beta!\gamma!}\frac{C_0^{|\beta|}C^{|\gamma|}|\gamma|!}
{\log^{|\gamma|}(|\gamma|+e)}\lesssim (\frac{1}{\log \kappa}+1) \frac{(|\alpha|+1)!  C^{|\alpha|}}{\log^{|\alpha|+1}(|\alpha|+1+e)}.
\end{align}
\end{lemma}
\begin{proof}
We shall reduce the proof to the one-dimensional case by the following combinatorial identity. For every $l,j\in \N$ such that $l+j=|\alpha|$, we have
\begin{align}\label{equ-high-bi-1}
\sum_{\beta,\gamma\in \N^d:|\beta|=l, |\gamma|=j,\beta+\gamma=\alpha} \frac{|\beta|!}{\beta!}\frac{|\gamma|!}{\gamma!}=\frac{|\alpha|!}{\alpha!}.
\end{align}
In fact, if $d=1$, \eqref{equ-high-bi-1} holds since both sides equal to $1$. Now assume $d\geq 2$. Fix $\alpha=(\alpha_1,\alpha_2,\ldots,\alpha_d)\in \N^d$. Consider a polynomial
$$
p(x_1,x_2,\ldots, x_{d})=(x_1+x_2+\cdots+x_d)^{|\alpha|}.
$$
On the one hand, expanding $p$ we find that the coefficient of $x^\alpha=x_1^{\alpha_1}x_2^{\alpha_2}\cdots x_d^{\alpha_d}$ is $\frac{|\alpha|!}{\alpha!}$, which is the right-hand side of \eqref{equ-high-bi-1}. On the other hand, we write
$$
p(x_1,x_2,\ldots, x_{d})=(x_1+x_2+\cdots+x_d)^{l}\cdot(x_1+x_2+\cdots+x_d)^{j}.
$$
For every multi-index $\beta\le \alpha$ with $|\beta|=l$, $\beta+\gamma=\alpha$, the expansion contains the term $\frac{|\beta|!}{\beta!}\frac{|\gamma|!}{\gamma!}x^\alpha$. Taking the sum for all $\beta\le \alpha$ satisfying $|\beta|=l$, we find that the left-hand side of \eqref{equ-high-bi-1} is also the coefficient of $x^\alpha$. This proves \eqref{equ-high-bi-1}.

With \eqref{equ-high-bi-1} in hand, we can rewrite the sum as
\begin{align*}
\sum_{\beta,\gamma\in \N^d: \beta+\gamma=\alpha}\frac{\alpha!}{\beta!\gamma!}
\frac{C_0^{|\beta|}C^{|\gamma|}|\gamma|!}{\log^{|\gamma|}(|\gamma|+e)}
&=\sum_{l+j=|\alpha|}\sum_{\substack{\beta,\gamma\in \N^d \\|\beta|=l,|\gamma|=j, \beta+\gamma=\alpha}}\frac{\alpha!}{\beta!\gamma!}
\frac{C_0^{|\beta|}C^{|\gamma|}|\gamma|!}{\log^{|\gamma|}(|\gamma|+e)}\\
&=\sum_{l+j=|\alpha|}
\frac{|\alpha|!C_0^{l}C^{j}  }{l!\log^{j}(j+e)}.
\end{align*}
Applying Lemma \ref{lem-bootstrap} to the right-hand side we obtain \eqref{equ-high-bi-0}.
\end{proof}

Fix $1<p<\infty$. By the elliptic regularity, we have
\begin{align}\label{equ-ana-high-1}
\sup_{|\alpha|=2}\|D^\alpha u\|_{L^p(\R^d)}\leq c_p\|\Delta u\|_{L^p(\R^d)}
\end{align}
for some constant $c_p>0$ depending only on $p,d$. Moreover, if $\kappa>1$ and $C\geq \kappa C_0$, it follows from Lemma \ref{lem-bootstrap-high} that there exists a numerical constant $K_1>0$ such that for all $\alpha\in \N^d$,
\begin{align}\label{equ-high-bi-1.5}
\sum_{\beta,\gamma\in \N^d: \beta+\gamma=\alpha}\frac{\alpha!}{\beta!\gamma!}\frac{C_0^{|\beta|}C^{|\gamma|}|\gamma|!}
{\log^{|\gamma|}(|\gamma|+e)}\leq K_1 (\frac{1}{\log \kappa}+1) \frac{(|\alpha|+1)!  C^{|\alpha|}}{\log^{|\alpha|+1}(|\alpha|+1+e)}.
\end{align}

Let $u\in L^p(\R^d)$ be a solution to \eqref{equ-elip} with $\|u\|_{L^p(\R^d)}=1$. We claim that  there exists a constant $K_2=K_2(d,p)$ depending only on $d,p$ such that
\begin{align}\label{equ-ana-high-2}
\sup_{|\alpha|=k}\|D^\alpha u\|_{L^p(\R^d)}\leq  \frac{C^{k}k!}{\log^{k}(k+e)}, \quad \forall k \in \N,
\end{align}
where $C=\kappa C_0+(K_2+c_p K_1 (d+1))(\frac{1}{\log \kappa}+1)$.

As before, we use the induction argument. Assume that \eqref{equ-ana-high-2} holds for all $k\leq n+1$ with some $n\geq 1$. Now we prove the same bound for $k=n+2$. Indeed, using \eqref{equ-ana-high-1} and the equation \eqref{equ-elip} we have
\begin{align*}
\sup_{|\alpha|=n+2}\|D^\alpha u\|_{L^p(\R^d)}&\leq c_p\sup_{|\alpha|=n} \|D^\alpha \Delta u\|_{L^p(\R^d)}
=c_p\sup_{|\alpha|=n}\|D^\alpha (W\cdot \nabla u +Vu)\|_{L^p(\R^d)}\\
&\leq c_p \sup_{|\alpha|=n} \sum_{\beta+\gamma=\alpha}\frac{\alpha!}{\beta!\gamma!}(\|D^\beta W\cdot D^\gamma\nabla u\|_{L^p(\R^d)}+\|D^\beta VD^\gamma u\|_{L^p(\R^d)})\\
&={\rm I}+{\rm II}.
\end{align*}
For the first term, using our assumptions on $W$ and \eqref{equ-ana-high-2} we get
\begin{align}\label{equ-99-1}
{\rm I}&\leq  c_pd \sup_{|\alpha|=n} \sum_{\beta+\gamma=\alpha}\frac{\alpha!}{\beta!\gamma!}C_0^{|\beta|}
\frac{C^{|\gamma|+1}(|\gamma|+1)!}{\log^{|\gamma|+1}(|\gamma|+1+e)}\nonumber\\
&\leq c_pd \sup_{|\alpha|=n} \sum_{\beta+\gamma=\alpha}\frac{(\alpha+e_1)!}{\beta!(\gamma+e_1)!}C_0^{|\beta|}
\frac{C^{|\gamma|+1}(|\gamma|+1)!}{\log^{|\gamma|+1}(|\gamma|+1+e)}\nonumber\\
&\leq  c_pd  \sup_{|\alpha|=n+1} \sum_{\beta+\gamma=\alpha}\frac{\alpha!}{\beta!\gamma!}C_0^{|\beta|}
\frac{C^{|\gamma|}|\gamma|!}{\log^{|\gamma|}(|\gamma|+e)}\nonumber\\
&\leq  c_pdK_1 (\frac{1}{\log \kappa}+1)    \frac{ C^{n+1}(n+2)!}{\log^{n+2}(n+{2}+e)},
\end{align}
where $e_1=(1,0,\ldots,0)$ and in the last step we used \eqref{equ-high-bi-1.5}.
Similarly, using \eqref{equ-high-bi-1.5} again we  have
$$
{\rm II}\leq  c_p K_1 (\frac{1}{\log \kappa}+1)   \frac{ C^{n}(n+1)!}{\log^{n+{1}}(n+1+e)}.
$$
Combining the bounds of $\rm I, II$, noting the definition of $C$, we deduce that
$$
\sup_{|\alpha|=n+2}\|D^\alpha u\|_{L^p(\R^d)}\leq  \frac{ C^{n+2}(n+2)!}{\log^{n+2}(n+2+e)}.
$$
Thus, \eqref{equ-ana-high-2} holds when $k=n+2$.

Now we check that \eqref{equ-ana-high-2} holds for all $k\leq 2$ with some well-chosen $K_2$. This is clear if $k=0$. To prove the bounds for $k=1,2$, we proceed as in the one-dimensional case. Owing to the elliptic regularity \eqref{equ-ana-high-1} and the equation \eqref{equ-elip}, we have
\begin{align}\label{equ-high-p-810-1}
\sup_{|\alpha|=2}\|D^\alpha u\|_{L^p(\R^d)}\lesssim_{d,p} \|W\cdot\nabla u+Vu\|_{L^p(\R^d)}
\lesssim_{d,p}  \sup_{|\alpha|=1}\|D^\alpha u\|_{L^p(\R^d)}+1.
\end{align}
Moreover, we have the interpolation inequality
\begin{align}\label{equ-high-p-810-2}
\sup_{|\alpha|=1}\|D^\alpha u\|_{L^p(\R^d)}  \lesssim_{d,p}\Big(\sup_{|\alpha|=2}\|D^\alpha u\|_{L^p(\R^d)} \|u\|_{L^p(\R^d)}\Big)^{1/2}.
\end{align}
It follows from \eqref{equ-high-p-810-1} and \eqref{equ-high-p-810-2} that $\sup_{|\alpha|\leq 2}\|D^\alpha u\|_{L^p(\R^d)}\lesssim_{d,p} 1$. Thus, \eqref{equ-ana-high-2} holds with some $K_2$ depending only on $p,d$. Finally, setting $K=K_2+c_pK_1(d+1)$, we complete the proof of Theorem \ref{thm-2} in the case when $1<p<\infty$.

\subsection{The case \texorpdfstring{$p=1$}{}}

Recall that for every $1\leq p\leq\infty, 1\leq q< \infty$, the amalgams of $L^p$ and $l^q$ is the Banach space of all functions $u\in L^p_{loc}$ for which the norm
$$
\|u\|_{l^qL^p(\R^n)}= \left( \sum_{j\in \Z^d} \|u\|^q_{L^p(\mathcal {C}(j))} \right)^{1/q}
$$
is finite, where $\mathcal {C}(j)$ is the unit cube (side length being $1$) in $\R^d$ centered at $j\in \Z^d$. If $q=\infty$, the space $\|u\|_{l^\infty L^p(\R^n)}$ can be defined similarly. We refer the reader to \cite{FournierAmalgamsLp1985} for the properties of $l^qL^p(\R^d)$.

We first claim that for every $1<p<\infty, 1\leq q\leq\infty$, the following elliptic regularity holds
\begin{align}\label{equ-ana-L1-812-1}
\sup_{|\alpha|=2}\|D^\alpha u\|_{l^qL^p(\R^d)}\lesssim_{d,p} \|\Delta u\|_{l^qL^p(\R^d)}+\| u\|_{l^qL^p(\R^d)}.
\end{align}
Indeed, using the classical interior regularity, we have for all $j\in \Z^d$,
$$
\sup_{|\alpha|=2}\|D^\alpha u\|_{L^p(B_{\sqrt{d}/2}(j))}\lesssim_{d,p} \|\Delta u\|_{L^p(B_{\sqrt{d}}(j))}+\| u\|_{L^p(B_{\sqrt{d}}(j))}.
$$
Noting that $\mathcal {C}(j)\subset B_{\sqrt{d}/2}(j)$ and $B_{\sqrt{d}}(j)\subset 2\sqrt{d}\mathcal {C}(j)$, we have
\begin{align*}
\sup_{|\alpha|=2}\|D^\alpha u\|_{L^p(\mathcal {C}(j))} &\lesssim_{d,p} \|\Delta u\|_{L^p(2\sqrt{d}\mathcal {C}(j))}+\| u\|_{L^p(2\sqrt{d}\mathcal {C}(j))}\\
&\lesssim_{d,p}\sup_{k\in \Z^d:|k-j|\leq d}\left(\|\Delta u\|_{L^p(\mathcal {C}(k))}+\| u\|_{L^p(\mathcal {C}(k))}\right).    
\end{align*}
The case $q=\infty$ is similar, we assume $1\leq q<\infty$. Taking $l^q$ norm on both sides, we obtain
$$
\sup_{|\alpha|=2} \|D^\alpha u\|_{l^qL^p(\R^d)}\lesssim_{d,p}  \left(\sum_{j}\sup_{k\in \Z^d:|k-j|\leq d} (\|\Delta u\|_{L^p(\mathcal {C}(k))}+\| u\|_{L^p(\mathcal {C}(k))})^q\right)^{1/q},
$$
where the right-hand side is bounded by a constant times $\|\Delta u\|_{l^qL^p(\R^d)}+\| u\|_{l^qL^p(\R^d)}$. This proves \eqref{equ-ana-L1-812-1}.

With \eqref{equ-ana-L1-812-1} in hand, noting the term $\| u\|_{l^qL^p(\R^d)}$ is not harmful in the iteration, one can proceed as in the previous section to prove the following result.
\begin{proposition}\label{prop-L1}
Let $p\in(1,\infty), q\in[1,\infty]$, and $u\in l^qL^p(\R^d)$ be a solution to \eqref{equ-elip} with $\|u\|_{l^qL^p(\R^d)}=1$.
Then there exists a constant $K=K(d,p)>0$ depending only on $d,p$ such that for any $\kappa>1$, $u$ satisfies
\begin{align}\label{equ-ultra-ana-lqLp}
 \|D^\alpha u\|_{l^qL^p(\R^d)}  \leq  \left(\kappa C_0+K(\frac{1}{\log \kappa}+1)  \right)^{|\alpha|} \frac{|\alpha|!}{\log^{|\alpha|}(|\alpha|+e)} , \quad \forall \alpha\in \N^d.
\end{align}
\end{proposition}

Now we prove Theorem \ref{thm-2} in the case $p=1$. We start with an interpolation inequality in \(L^1(\R^d)\). More preciesly, for any \(\varepsilon>0\) and all \(u\in L^1(\R^d), \Delta u\in L^1(\R^d)\), we have
\begin{align}\label{equ-99-9}
 \|\nabla u\|_{L^1(\R^d)}\lesssim_d \varepsilon \|\Delta u\|_{L^1(\R^d)}+ \frac{1}{\varepsilon}\|u\|_{L^1(\R^d)}.   
\end{align}
By a scaling \(u(x)\to u(\varepsilon x)\), it suffices to show that
\begin{align}\label{equ-99-10}
 \|\nabla u\|_{L^1(\R^d)}\lesssim_d \|(1-\Delta) u\|_{L^1(\R^d)}.   
\end{align}
To this end, let \(f=(1-\Delta) u\), then 
$$
\nabla u = \nabla G_2 * f, 
$$
where $G_s(x)$ is the integral kernel associated with $(1-\Delta)^{-s/2}$ with $s>0$, namely $(1-\Delta)^{-s/2}u=G_s * u$ for all $u$. Using the Fourier transform, $G_s$ has the representation
\begin{align}\label{equ-Gs-1}
G_s(x)=\frac{(2 \sqrt{\pi})^{-d}}{\Gamma\left(\frac{s}{2}\right)} \int_0^{\infty} e^{-t} e^{-\frac{|x|^2}{4 t}} t^{\frac{s-d}{2}} \frac{d t}{t}.
\end{align}
Moreover, one can show that $G_s(x)>0$ for all $x\in \R^d$, $G_s(x)\lesssim_{s,d} e^{-\frac{|x|}{2}} $ if $|x|\geq 2$ and
\begin{align}\label{equ-Gs-loc}
 G_s(x)\lesssim_{s,d}
\left\{
\begin{array}{ll}
 1+|x|^{s-d}, \quad & 0<s<d,\\
1+\log \frac{2}{|x|}, \quad & s=d,\\
1, \quad &s>{d}
\end{array}
\right.
\end{align}
if $|x|<2$; see \cite[Proposition 6.1.5]{GrafakosModernFourier2009} for a proof. It follows from \eqref{equ-Gs-1} that
\begin{align*}
|\nabla G_s(x)|\lesssim_{s,d} \int_0^{\infty} e^{-t} e^{-\frac{|x|^2}{4 t}} \frac{|x|}{2t} t^{\frac{s-d}{2}} \frac{d t}{t}\lesssim_{s,d} \int_0^{\infty} e^{-t} e^{-\frac{|x|^2}{8 t}}   t^{\frac{s-1-d}{2}} \frac{d t}{t}\lesssim_{s,d} G_{s-1}(\frac{x}{\sqrt{2}})
\end{align*}
for all $x\in \R^d$, where we used $e^{-\frac{|x|^2}{8 t}} \frac{|x|}{2t^{1/2}}\lesssim 1$. Thus, according to \eqref{equ-Gs-loc}, we have $\nabla G_2\in L^1(\R^d)$. Thus by Young's inequality of convolution,
$$
\|\nabla u\|_{L^1(\R^d)}=\|\nabla G_2 * f\|_{L^1(\R^d)}\leq \|\nabla G_2\|_{L^1(\R^d)} \|f\|_{L^1(\R^d)}\lesssim_d \|f\|_{L^1(\R^d)},
$$
which proves \eqref{equ-99-10}.

\medskip
  Let $u\in L^1(\R^d)$ be a solution to \eqref{equ-elip} with $\|u\|_{L^1(\R^d)}=1$. 
Using the equation \eqref{equ-elip} and the fact $\|W\|_{L^\infty(\R^d)}, \|V\|_{L^\infty(\R^d)}\leq 1$, we find 
$$
\|\Delta u\|_{L^1(\R^d)}\leq \|\nabla u\|_{L^1(\R^d)}+1.
$$
This, together with the interpolation inequality {\eqref{equ-99-9}} with small $\varepsilon>0$, implies that $\|\Delta u\|_{L^1(\R^d)}\lesssim_d 1$ and thus
\begin{align}\label{equ-99-11}
\|(1-\Delta) u\|_{L^1(\R^d)}\lesssim_d 1.
\end{align} 
Moreover, from \eqref{equ-Gs-loc}, there exists some $p_0\in(1,\infty)$ such that $G_2\in l^1L^{p_0}(\R^d)$. Owing to Young's inequality in $l^qL^p(\R^d)$, we deduce from \eqref{equ-99-11} that
$$
\|u\|_{l^1L^{p_0}(\R^d)}=\|G_2*(1-\Delta)u\|_{l^1L^{p_0}(\R^d)}\lesssim_d \|G_2\|_{l^1L^{p_0}(\R^d)}\|(1-\Delta)u\|_{L^1(\R^d)}  \lesssim_d 1.
$$
Now we apply Proposition \ref{prop-L1} to conclude that there exists a constant $K=K(d)>0$, depending only on $d$, such that for any $\kappa>1$, the solution satisfies 
\begin{align*}
 \|D^\alpha u\|_{l^1L^{p_0}(\R^d)}  &\leq \|u\|_{l^1L^{p_0}(\R^d)} \left(\kappa C_0+K(\frac{1}{\log \kappa}+1)  \right)^{|\alpha|} \frac{|\alpha|!}{\log^{|\alpha|}(|\alpha|+e)}\\
 &\leq c(d) \left(\kappa C_0+K(\frac{1}{\log \kappa}+1)  \right)^{|\alpha|} \frac{|\alpha|!}{\log^{|\alpha|}(|\alpha|+e)}
\end{align*}
for all $\alpha\in \N^d$. Note that by H\"{o}lder inequality, we have $\|D^\alpha u\|_{L^1(\R^d)}\leq \|D^\alpha u\|_{l^1L^{p_0}(\R^d)}$. Thus, $\|D^\alpha u\|_{L^1(\R^d)}$ satisfies the same upper bound. Choosing a new constant $K$ completes the proof in the case $p=1$.

\subsection{The case \texorpdfstring{$p=\infty$}{}}

In this case, the maximal regularity of the second-order elliptic equations fails, thus the iteration scheme used in previous subsection does not work now. To overcome this difficulty, we shall use Schauder estimates and iterate in H\"{o}lder spaces instead. We first recall several facts of H\"{o}lder spaces, which can be found in Krylov \cite{KrylovHolder1996}.

For $k=0,1,2, \ldots$, we denote by $C_{{loc }}^k(\R^d)$ the set of all functions $u=u(x)$ whose derivatives $D^\alpha u$ for $|\alpha| \leq k$ are continuous in $\R^d$. We set
$$
|u|_{0 ; \R^d}=[u]_{0 ; \R^d}=\sup _{\R^d}|u|, \quad[u]_{k ; \R^d}=\max _{|\alpha|=k}\left|D^\alpha u\right|_{0 ; \R^d} .
$$

 For $k=0,1,2, \ldots$ the space $C^k(\R^d)$ is the Banach space of all functions $u \in C_{loc}^k(\R^d)$ for which the following norm
$$
|u|_{k ; \R^d}=\sum_{j=0}^k[u]_{j ; \R^d}
$$
is finite. If $0<\delta<1$, we call $u$ H\"{o}lder continuous with exponent $\delta$ in $\R^d$ if the seminorm
$$
[u]_{\delta ; \R^d}=\sup _{\substack{x, y \in \R^d \\ x \neq y}} \frac{|u(x)-u(y)|}{|x-y|^\delta}
$$
is finite. This seminorm is called H\"{o}lder's constant of $u$ of order $\delta$.   We set
$$
[u]_{k+\delta ; \R^d}=\max _{|\alpha|=k}\left[D^\alpha u\right]_{\delta ; \R^d} .
$$

 For $0<\delta<1$ and $k=0,1,2, \ldots$ the H\"{o}lder space $C^{k+\delta}(\R^d)$ is the Banach space of all functions $u \in C^k(\R^d)$ for which the norm
$$
|u|_{k+\delta ; \R^d}=|u|_{k ; \R^d}+[u]_{k+\delta ; \R^d}
$$
is finite.

For $0<\delta<1$ and all functions $u,v\in C^\delta(\R^d)$,
\begin{align}\label{equ-hold-prod}
[uv]_{\delta ; \R^d}\leq [u]_{\delta ; \R^d}|v|_{0 ; \R^d}+|u|_{0 ; \R^d}[v]_{\delta ; \R^d}.
\end{align}
For all $0\leq s\leq r$, we have the interpolation inequality  in H\"{o}lder spaces
\begin{align}\label{equ-hold-inter}
[u]_{s ; \R^d}\lesssim_{r,d} \varepsilon^{r-s}[u]_{r ; \R^d}+\varepsilon^{-s}|u|_{0 ; \R^d}, \quad \forall \varepsilon>0;
\end{align}
see \cite[Theorem 3.2.1]{KrylovHolder1996}. The Schauder estimate (see \cite[Theorem 3.4.1]{KrylovHolder1996}) reads that
\begin{align}\label{equ-schauder}
[u]_{2+\delta ; \R^d}\lesssim_d [\Delta u]_{\delta ; \R^d}.
\end{align}

From now on, we fix $\delta=\frac{1}{2}$.

Let $u\in L^\infty(\R^d)$ be a solution of \eqref{equ-elip} with $\|u\|_{L^\infty(\R^d)}=1$.
Now we  show that there exist  constants $K_2, K_3>0$ depending only on $d$  such that
\begin{align}\label{equ-ana-high-infty-1}
 [ u]_{k+\delta ; \R^d}\leq K_2\frac{C^{k}k!}{\log^{k}(k+e)}, \quad \forall k\in \N,
\end{align}
where $C=\kappa C_0+ K_3K_1 (\frac{1}{\log \kappa}+1)$ and $K_1$ is the same as that in \eqref{equ-high-bi-1.5}.

Assuming that \eqref{equ-ana-high-infty-1} holds for all $k\leq n+1$ with some $n\geq 1$, we show that it also holds for $k=n+2$. To this end,  we first give a bound of $[u]_{k; \R^d}$  under the assumption \eqref{equ-ana-high-infty-1}. We could not apply \eqref{equ-hold-inter} directly, since the implicit constant may depend on $r$. However, a bound of \([u]_{k; \R^d}\) can be obtained by interpolating between \([u]_{k+\delta; \R^d}\) and \([u]_{k-1+\delta; \R^d}\). This is done in the following lemma.


\begin{lemma}\label{lem-ana-high-infty-1}
If \eqref{equ-ana-high-infty-1} holds for all $k\leq n+1$, then
\begin{align}\label{equ-ana-high-infty-3}
[u]_{k; \R^d}\lesssim_d   \frac{C^{k-\frac{1}{2}}k!}{k^{\frac{1}{2}}\log^{k-\frac{1}{2}}(k+e)}, \quad \forall 1\leq k\leq n+1.
\end{align}
\end{lemma}
\begin{proof}
Fix $1\leq k\leq n+1$ and $\alpha\in \N^d$ with $|\alpha|=k$.  Let $x\in \R^d$ and $\eta\in C_c^\infty(B_1(x))$ be a non-negative {smooth} function such that $\int_{\R^d}\eta(y)\d y =1$. For every $r>0$, set $\eta_r(y)=r^{-d}\eta(y/r), y\in \R^d$. Then
\begin{align}\label{equ-98-1}
  D^\alpha u(x)=\int_{\R^d}\Big( D^\alpha u(x)-D^\alpha u(x+y)\Big) \eta_r(y) \d y+   \int_{\R^d} D^\alpha u(x+y)  \eta_r(y) \d y.
\end{align}
The first term on the right-hand side of \eqref{equ-98-1} is bounded by
\begin{align}\label{equ-98-2}
\int_{\R^d}[D^\alpha u]_{\delta;\R^d}r^\delta \eta_r(y) \d y\leq [u]_{k+\delta;\R^d}r^\delta.    
\end{align}
Let $e_j$ be the unit vector in the $j$-th direction in $\R^d$. Using integration by parts to the second term {and the fact that $\int_{\R^d} \partial_j\eta_r(y) \d y=0$}, we have for $1\leq j\leq d$,
\begin{align*}
 \int_{\R^d} D^\alpha u(x+y)  \eta_r(y) \d y &= -\int_{\R^d} D^{\alpha-e_j} u(x+y)  \partial_j\eta_r(y) \d y\\
 &=\int_{\R^d} \Big(D^{\alpha-e_j} u(x)-D^{\alpha-e_j} u(x+y)\Big)  \partial_j\eta_r(y) \d y. 
\end{align*}
It follows that
\begin{align}\label{equ-98-3}
\left|\int_{\R^d} D^\alpha u(x+y)  \eta_r(y) \d y \right|
\leq  \int_{\R^d} [u]_{k-1+\delta;\R^d}r^{\delta} |\partial_j\eta_r(y)| \d y\lesssim_d [u]_{k-1+\delta;\R^d}r^{-\delta}
\end{align}
since $\|\partial_j\eta_r\|_{L^1(\R^d)}\lesssim_d r^{-1}$ and $\delta=\frac{1}{2}$. 

Combining \eqref{equ-98-1}-\eqref{equ-98-3} we obtain
$$
|D^\alpha u(x)|\lesssim_d [u]_{k+\delta;\R^d}r^\delta+ [u]_{k-1+\delta;\R^d}r^{-\delta}.   
$$
Optimizing with respect to $r>0$ and using \eqref{equ-ana-high-infty-1},  we get
\begin{align*}
|D^\alpha u(x)|\lesssim_d ([u]_{k+\delta;\R^d}[u]_{k-1+\delta;\R^d})^{\frac{1}{2}}
\lesssim_d  \frac{C^{k-\frac{1}{2}}k!}{k^{\frac{1}{2}}\log^{k-\frac{1}{2}}(k+e)}.
\end{align*}
Taking the supremum with $x$ gives \eqref{equ-ana-high-infty-3}.
\end{proof}

Now we estimate $[u]_{n+2+\delta; \R^d}$.
By the Schauder estimate \eqref{equ-schauder} and the equation \eqref{equ-elip}, we have
\begin{align*}
  [u]_{n+2+\delta; \R^d} &\leq c_d  [\Delta u]_{n+\delta; \R^d}= c_d\sup_{|\alpha|=n}[D^\alpha(W\cdot \nabla u +Vu)]_{\delta; \R^d}\\
 &\leq \sup_{|\alpha|=n}  \sum_{\beta+\gamma=\alpha}\frac{\alpha!}{\beta!\gamma!}\Big(  [D^\beta W \cdot D^\gamma \nabla u]_{\delta; \R^d} + [D^\beta V D^\gamma u]_{\delta; \R^d}\Big)={\rm I}+{\rm II}.
\end{align*}
To bound ${\rm I}$, we note that
\begin{align}\label{equ-ana-high-infty-6}
[D^\beta W ]_{\delta; \R^d}\lesssim_d C_0^{|\beta|+\frac{1}{2}}, \quad \forall \beta\in \N^d.
\end{align}
In fact, by the definition, for any $\varepsilon>0$
$$
[D^\beta W]_{\delta; \R^d}\leq \sup _{\substack{x, y \in \R^d \\ 0<|x-y|\leq \varepsilon}} \frac{|D^\beta W(x)-D^\beta W(y)|}{|x-y|^\delta} + \sup _{\substack{x, y \in \R^d \\ |x-y|>\varepsilon}} \frac{|D^\beta W(x)-D^\beta W(y)|}{|x-y|^\delta}.
$$
The first term is bounded by $d\varepsilon^{1-\delta}C_0^{|\beta|+1}$ by the mean value theorem, while the second term is bounded by $2C_0^{|\beta|}\varepsilon^{-\delta}$. Optimizing with respect to $\varepsilon$ gives \eqref{equ-ana-high-infty-6}.

Using \eqref{equ-hold-prod}, \eqref{equ-ana-high-infty-1}, \eqref{equ-ana-high-infty-3}, and \eqref{equ-ana-high-infty-6}, we get
\begin{align*}
{\rm I} &\lesssim_d\sup_{|\alpha|=n}  \sum_{\beta+\gamma=\alpha}\frac{\alpha!}{\beta!\gamma!}\Big([D^\beta W]_{\delta; \R^d}|D^\gamma \nabla u|_{0; \R^d}+|D^\beta W|_{0; \R^d}[D^\gamma \nabla u]_{\delta; \R^d}  \Big)\\
&\lesssim_d K_2 \sup_{|\alpha|=n}  \sum_{\beta+\gamma=\alpha}\frac{\alpha!}{\beta!\gamma!}\left( C_0^{|\beta|+\frac{1}{2}}  \frac{C^{|\gamma|+\frac{1}{2}}(|\gamma|+1)!}{\sqrt{|\gamma|+1}\log^{|\gamma|+\frac{1}{2}}(|\gamma|+1+e)} + C_0^{|\beta|}  \frac{C^{|\gamma|+1}(|\gamma|+1)!}{\log^{|\gamma|+1}(|\gamma|+1+e)}  \right)\\
&\lesssim_d K_2 \sup_{|\alpha|=n}  \sum_{\beta+\gamma=\alpha}\frac{\alpha!}{\beta!\gamma!} C_0^{|\beta|}  \frac{C^{|\gamma|+1}{(|\gamma|+1)!}}{\log^{|\gamma|{+1}}(|\gamma|+1+e)} 
\end{align*}
since $C_0\leq C$. Thus, similar to \eqref{equ-99-1}, we have
$$
I\lesssim K_1K_2(\frac{1}{\log \kappa}+1)  \frac{(n+2)!  C^{n+1}}{\log^{n+2}(n+2+e)}.
$$

 Similarly, the second term $\rm II$ has the same upper bound. Adding the estimates of $\rm I$ and $\rm II$ together, we obtain
\begin{align*}
 [u]_{n+2+\delta; \R^d}\leq c_d  K_1K_2(\frac{1}{\log \kappa}+1) \frac{ C^{n+1} (n+2)!}{ \log^{n+2}(n+2+e)}
\end{align*}
for some constant $c_d>0$ depending only on $d$.
Setting $K_3=c_d  K_1$, we conclude \eqref{equ-ana-high-infty-1} with $k=n+2$.

It remains to show \eqref{equ-ana-high-infty-1} for $k\leq 2$. Since $u\in L^\infty(\R^d)$ is a solution of \eqref{equ-elip} with $\|u\|_{L^\infty(\R^d)}=1$, the Schauder estimate \eqref{equ-schauder} gives
\begin{align*}
 [u]_{2+\delta ; \R^d}&\lesssim_d[\Delta u]_{\delta ; \R^d} = [W\cdot \nabla u+Vu]_{\delta ; \R^d}\\
 &\lesssim_d |W|_{0;\R^d}[u]_{1+\delta;\R^d}+[W]_{\delta;\R^d}[u]_{1;\R^d}+|V|_{0;\R^d}[u]_{\delta;\R^d}+[V]_{\delta;\R^d}[u]_{0;\R^d}\\
 &\lesssim_d
[u]_{1+\delta;\R^d}+C_0^{\frac{1}{2}}[u]_{1;\R^d}+ [u]_{\delta;\R^d}+C_0^{\frac{1}{2}}. 
\end{align*}
Moreover, by the interpolation inequality \eqref{equ-hold-inter}, we bound the terms on the right-hand side as
$$
 [u]_{1+\delta ; \R^d}\lesssim_d \varepsilon[u]_{2+\delta ; \R^d} +\varepsilon^{-1-\delta}, \quad [u]_{\delta ; \R^d}\lesssim_d \varepsilon^2[u]_{2+\delta ; \R^d} +\varepsilon^{-\delta}
$$
and with a scaling
$$
C_0^{\frac{1}{2}}[u]_{1;\R^d}\lesssim_d \varepsilon^{1+\delta}[u]_{2+\delta;\R^d}+\varepsilon^{-1}C_0^{\frac{1}{2}(1+\frac{1}{1+\delta})}
$$
for any $\varepsilon>0$. Taking $\varepsilon>0$ small enough, we get
$$
[u]_{2+\delta ; \R^d}\lesssim_d C_0^{\frac{1}{2}}+C_0^{\frac{1}{2}(1+\frac{1}{1+\delta})}\lesssim 1+C_0.
$$
Then, we conclude that $[u]_{k+\delta ; \R^d}\lesssim_d 1+C_0$ for all $k\leq 2$. Thus, \eqref{equ-ana-high-infty-1} holds with some large $K_2$ depending only on $d$ for all $k\in \N$.

Finally, by \eqref{equ-ana-high-infty-1} and Lemma \ref{lem-ana-high-infty-1}, we have
$$
[u]_{k; \R^d}\lesssim_d   \frac{C^{k-\frac{1}{2}}k!}{\sqrt{k}\log^{k-\frac{1}{2}}(k+e)}\lesssim \frac{C^{k-\frac{1}{2}}k!}{\log^{k}(k+e)}, \quad \forall 1\leq k\in \N,
$$
where $C=\kappa C_0+ K_3K_1 (\frac{1}{\log \kappa}+1)$. Choosing another constant $K_3$, we find
$$
[u]_{k; \R^d}\leq \frac{C^{k}k!}{\log^{k}(k+e)}, \quad \forall k\in \N.
$$
This completes the proof in the case when $p=\infty$.



\bibliographystyle{siam} %

\end{document}